\documentclass[12pt]{article}
\usepackage{amsmath,latexsym,amssymb,amsfonts,amsbsy, amsthm}
\usepackage[usenames]{color}
\usepackage{xspace,colortbl}
\allowdisplaybreaks

\newcommand{\beq}{\begin{equation}}
\newcommand{\eeq}{\end{equation}}
\newcommand{\ben}{\begin{eqnarray}}
\newcommand{\een}{\end{eqnarray}}
\newcommand{\beno}{\begin{eqnarray*}}
\newcommand{\eeno}{\end{eqnarray*}}

\newtheorem{thm}{Theorem}[section]

\newtheorem{lem}[thm]{Lemma}
\newtheorem{prop}[thm]{Proposition}

\newtheorem{rmk}[thm]{Remark}

\newcommand{\dif}{\mathrm{d}}
\newcommand{\me}{\mathrm{e}}
\newcommand{\mi}{\mathrm{i}}

\newcommand{\lag}{\langle}
\newcommand{\rag}{\rangle}

\begin{document}
\topmargin -9mm  \oddsidemargin -1mm
\title{\bf Periodic solutions to Klein-Gordon systems with linear couplings
\author{Jianyi Chen $^{a,}$\thanks{Corresponding author.
\newline\indent The reaserch was supported by National Natural Science Foundation of China (11701310, 11771428,12031015,12026217), Natural Science Foundation of Shandong Province (ZR2016AQ04), and the Research Foundation for Advanced Talents of Qingdao Agricultural University (6631114328).
\newline\indent
\textit{E-mail addresses}: chenjy@amss.ac.cn (J.Y. Chen); zzt@math.ac.cn (Z.T. Zhang); lucycgj@163.com (G.J. Chang); zhaojinglzbzj@163.com (J. Zhao).}\ ,\  \  Zhitao Zhang $^{b,c}$,
\ \ Guijuan Chang $^{a}$,\ \ Jing Zhao $^{a}$\\
{\small $^{a}$ Science and Information College, Qingdao Agricultural University, }\\
{\small Qingdao 266109, P. R. China
 }\\
{\small $^{b}$ Academy of Mathematics and Systems Science, Chinese Academy }\\
{\small  of Sciences, Beijing 100190, P. R. China;}\\
{\small   $^c$ School of Mathematical Sciences, University of Chinese  Academy  }\\
{\small of Sciences, Beijing 100049, P. R. China}
} }
\date{}
\maketitle

\renewcommand{\theequation}{\thesection.\arabic{equation}}
\setcounter{equation}{0}

\begin{abstract}
 In this paper, we study the  nonlinear Klein-Gordon systems arising from relativistic physics and quantum field theories
  $$\left\{\begin{array}{lll}
  u_{tt}- u_{xx} +bu + \varepsilon v + f(t,x,u) =0,\;\\
  v_{tt}- v_{xx} +bv + \varepsilon u + g(t,x,v) =0
  \end{array}\right.
  $$
 where $u,v$ satisfy  the Dirichlet boundary conditions on spatial interval $[0, \pi]$,  $b>0$ and $f$, $g$ are $2\pi$-periodic in $t$. We are concerned with the existence, regularity and asymptotic behavior of time-periodic solutions to the linearly coupled problem as $\varepsilon$ goes to 0. Firstly, under some superlinear growth and monotonicity assumptions on $f$ and $g$, we obtain the solutions $(u_\varepsilon, v_\varepsilon)$ with time-period $2\pi$ for the problem as the linear coupling constant $\varepsilon$ is sufficiently small, by constructing critical points of an indefinite functional via variational methods. Secondly, we give precise characterization for the asymptotic behavior of these solutions, and show that as $\varepsilon\rightarrow 0$, $(u_\varepsilon, v_\varepsilon)$ converge to the solutions of the wave equations without the coupling terms. Finally, by careful analysis which are quite different from the elliptic regularity theory, we obtain some interesting results concerning the higher regularity of the periodic solutions.
\end{abstract}

\renewcommand{\theequation}{\thesection.\arabic{equation}}
\setcounter{equation}{0}

\section{ Introduction}
\label{} \ \ \ \ \
    In this paper, we consider the important nonlinear Klein-Gordon system
\begin{equation*}
  \qquad\qquad\quad\ \left\{\begin{array}{l}
     u_{tt}- u_{xx} +bu + \varepsilon v + f(t,x,u) =0,\;\  t\in {\bf{\mathbb{R}}},\ x\in [0, \pi]
    \\[1.2ex]
     v_{tt}- v_{xx} +bv + \varepsilon u + g(t,x,v) =0,\;\  t\in {\bf{\mathbb{R}}},\ x\in [0, \pi]
 \end{array}\right.  \qquad\qquad\qquad  (1.1)_{a}
  \end{equation*}
  satisfying the Dirichlet boundary value conditions on the $x$-axis
  $$u(t,0)=u(t, \pi)=0, \;\ \ v(t,0)=v(t, \pi)=0, \;\ \  t\in \mathbb{R},\qquad \eqno (1.1)_{b}
  $$
 and periodic conditions with respect to the time variable $t$
 $$u(t+2\pi,x)=u(t,x),\;\ \ v(t+2\pi,x)=v(t,x),\;\ \  t\in \mathbb{R},\ x\in [0, \pi],  \eqno (1.1)_{c}
 $$
 where $u(t,x)$ and $v(t,x)$ are the relativistic wave functions generated by the interaction of two mass fields, $b>0$ and $\sqrt b$ stands for the mass; $\varepsilon$ denotes the strength of the fields coupling, and $\varepsilon$ is assumed to be sufficiently small. The nonlinear forced terms $f(t,x,u)$, $g(t,x,v)$ are $2\pi$-periodic in $t$. We study the existence, regularity and asymptotic behavior of time-periodic solutions to the linearly coupled problem (1.1)$_ {a, b, c}$ as $\varepsilon$ goes to 0. \par

   The coupled Klein-Gordon system in the general form
$$ u_{tt} - u_{xx} = H_{u}(u, v),\ \ \  v_{tt} - v_{xx} = H_{v}(u, v), \eqno {\rm (KG)}
$$
is deeply connected with many branches of mathematical physics, such as relativistic physics and quantum field theories. For instance, with the proper choice of the potential function $H(u, v)$, the system (KG) was used to describe the long-wave dynamics
of two coupled one-dimensional periodic chains in the bi-layer materials or the spinless relativisitic composite particles (see \cite{AkBaKh, KhPe}). Moreover, variations of such systems were also proposed in the work of Klainerman and Tataru \cite{KlTa} as important models to investigate the Yang-Mills equations under the Coulomb gauge condition. The solvability of (KG) depends upon the nature of the nonlinearities and the type of the boundary conditions. Many interesting theoretical and numerical results can be found in \cite{BerMus, FaXu, Kim, NiWa, Suna, WaIsWu, WuWaSh} and the monograph of Shatah-Struwe \cite{ShSt} which contains more extensive references.    \par
     It is an important work to study the existence and regularity of time-periodic solutions for the Dirichlet problem of (KG) with the gradient of the potential function $H(u, v)$ having the interesting coupling form
  $$ \nabla H(u,v) = \left(-bu - \varepsilon v - f(t,x,u),\ \ -bv - \varepsilon u - g(t,x,v) \right).
  $$
 When $\varepsilon =0$, Eq. (1.1)$_a$ are two copies of nonlinear wave equations
   $$ u_{tt} - u_{xx} + bu + f(t,x,u) =0,\;\ \ \ t\in {\bf{\mathbb{R}}},\ \ 0<x<\pi, \eqno \rm{(W1)}
   $$
   $$v_{tt} - v_{xx} + bv + g(t,x,v) =0,\;\ \ \ t\in {\bf{\mathbb{R}}},\ \ 0<x<\pi.     \eqno  \rm{(W2)}
   $$
\par
It is well known that even the existence of time-periodic solutions for single wave equation is difficult to study. Since the  seminal work \cite{Ra} of P. H. Rabinowitz, several tools in nonlinear analysis are developed by H. Br\'{e}zis, L. Nirenberg, J. M. Coron, K. C. Chang, J. Mawhin, M. Schechter, S. J. Li et al. to obtain the existence and multiplicity results of the periodic solutions for the scalar wave equations with various type of nonlinearities. We refer to [5-8, 14, 16, 17, 21, 25, 28, 29, 35]  for one dimensional problem, and [9-11, 26, 30, 31] for higher dimensional cases.
 The existence of solutions with time-period $T$ to such kinds of wave equations depends upon the nature of the parameter $b$, period $T$ and the nonlinearities.  All the above results require the crucial condition that $T/\pi$ is rational and sometimes take $T=2\pi$ for simplicity. When $T$ is an irrational multiple of $\pi$, we are led to the problems of small divisors which are difficult to deal with (see \cite{BeBo, Bou, Way} for examples).

 As $\varepsilon \neq 0$, the solvability of the system (1.1)$_ {a, b, c}$ is more complicated because of the presence of the linear coupling terms. We need a more delicate analysis to study the behavior of the interaction between the two linear coupling terms. Berkovits and Mustonen \cite{BerMus} used the topological degree theory and continuation principle to obtain at least one weak solution $(u, v)$ with time-period $2\pi$ for the system
  \begin{equation*}
 \qquad\qquad\ \  \left\{\begin{array}{l}
     u_{tt}- u_{xx}  + \lambda v + f(t,x,u) =h_1 (t,x),\;\  t\in {\bf{\mathbb{R}}},\ x\in [0, \pi]
    \\[1.2ex]
     v_{tt}- v_{xx} + \mu u + g(t,x,v) =h_2 (t,x),\;\  t\in {\bf{\mathbb{R}}},\ x\in [0, \pi]
 \end{array} \right.\qquad\qquad (KG)_{\lambda\mu}
  \end{equation*}
 where $\lambda\mu <0$, $h_1$, $h_2 \in L^2([0,2\pi]\times [0,\pi])$ and $f$, $g$ satisfy some linear growth
conditions. The assumption $\lambda\mu <0$ required in \cite{BerMus} plays a crucial role in calculating the degree and getting a priori
bounds of solutions for the corresponding homotopy equations. \par
 Recently, Yan, Ji and Sun \cite{YJS} used the change of degree argument to prove the existence of time-periodic weak solutions for some coupled Klein-Gordon systems with variable coefficients when the forced terms satisfy some sublinear conditions.

 To the best of our knowledge, little further progress has been made on the study of the existence and regularity of periodic solutions for (KG)$_{\lambda\mu}$ with superlinear forced terms.
  In the present work, we study such a superlinear problem and consider the situation $\lambda\mu >0$ by variational methods. We focus on the  case of $\lambda = \mu$ because the variational structure is required. Our results are in three aspects:\par
   $\bullet$ existence of the time-periodic weak solutions for (1.1)$_ {a, b, c}$,\par
    $\bullet$ asymptotic behavior of the weak solutions as $\varepsilon \rightarrow 0$,\par
    $\bullet$ higher regularity of the solutions.
\par\bigskip
Let $\Omega = [0, 2\pi]\times [0, \pi]$, we say that $(u, v)\in L^2(\Omega)\times L^2(\Omega)$ is a weak solution to the system (1.1)$_ {a, b, c}$ provided that
$$\iint_{\Omega}\big[u\big(\varphi_{tt}-\varphi_{xx}\big)+ bu\varphi + \varepsilon v\varphi +
 f(t,x,u)\varphi\big]~\dif t\dif x=0
 $$
and
  $$\iint_{\Omega}\big[v\big(\psi_{tt}-\psi_{xx}\big)+ bv\psi + \varepsilon u\psi +
 g(t,x,v)\psi\big]~\dif t\dif x=0
 $$
for all functions $\varphi$ and $\psi$ satisfying the conditions (1.1)$_ {b, c}$ and belonging to the space $H$ which is defined by (2.1) in Sect. 2.\par
  Let $\sigma(L)$ be the set of eigenvalues of the d'Alembert operator $L= \partial_{t}^2 - \partial_{x}^2$  subject to the conditions
(1.1)$_ {b, c}$, and denote by $\ker L$ the kernel of the operator $L$. It is well known that $\sigma(L)$ consists of the isolate numbers $\lambda_{jk}=j^2 -k^2$, for  $j \in \mathbb{Z}_{+}$ and $k \in \mathbb{Z}$. We see: \par
 (i) $0\in \sigma(L)$, and the multiplicity of eigenvalue $\lambda_{j_0 k_0}=0$ is infinite since there exists an infinite number of $j_0$ and $k_0$ such that $j_0^2 - k_0^2 =0$; that is, $\ker L$ is an infinite dimensional space; \par
 (ii) all the nonzero eigenvalues of $L$ are of finite multiplicity, and they tend to $+\infty$ or $-\infty$;\par
 (iii) for any number $b$ satisfying $-b\not\in \sigma(L)$, there exists a constant $\eta >0$ such that
   $$|\lambda_{jk} + b|\geq \eta, \ \ \forall~ j \in \mathbb{Z}_{+},\ k \in \mathbb{Z},  \eqno(1.2)
   $$
noting that the eigenvalues $\lambda_{jk}$ of the operator $L$ are isolated.\par\medskip
   {\bf (a) Existence and the asymptotic behavior of the solutions for (1.1)$_ {a, b, c}$}\par\medskip
 The first result of this paper is the following theorem concerning the existence of the time-periodic weak solutions for (1.1)$_ {a, b, c}$.

\begin{thm}\label{thmground}
Let $b>0$ and $-b\not\in \sigma(L)$, $f$, $g\in C (\Omega \times \mathbb{R},~\mathbb{R})$ are assumed to
be $2\pi$-periodic in $t$ and satisfy the following superlinear growth and monotonicity conditions {\rm(h1)}-{\rm(h4):}\par
{\rm(h1)} there exist $p>1$, $q>1$, and $c_0 >0$, such that for all $t$, $x$, $\xi$,
  $$|f(t,x,\xi)|\leq c_0 (1+|\xi|^p), \ {\rm{and}}\  |g(t,x,\xi)|\leq c_0 (1+|\xi|^q){\rm ;}
  $$
\indent {\rm(h2)}
 $f(t,x,\xi)=o(|\xi|)\ {\rm and}\ g(t,x,\xi)=o(|\xi|),\ \ {\rm as}\ \xi \rightarrow 0\ {\rm uniformly\ in}\ (t,x){\rm ;}
 $
 \\
\indent {\rm(h3) (Ambrosetti-Rabinowitz condition)}
$$(p+1) F(t,x,\xi) \leq f(t,x,\xi)\xi,\ {\rm and} \ (q+1) G(t,x,\xi) \leq g(t,x,\xi)\xi,
$$
for all $t$, $x$ and $\xi$, where $F(t,x,\xi)= \int_{0}^{\xi}f(t,x,s)\dif s$ and $G(t,x,\xi)= \int_{0}^{\xi}g(t,x,s)\dif s${\rm ;}\par
{\rm(h4) (Monotonicity)} $f(t,x,\xi)$ and $g(t,x,\xi)$ are nondecreasing in $\xi$.\par
  Then there exists $\varepsilon_0 >0$ such that for $|\varepsilon| < \varepsilon_0$, the system {\rm(1.1)}$_ {a, b, c}$ has at least one nontrivial weak solution $(u,v)\in L^2 (\Omega) \times L^2 (\Omega)$ with time period $2\pi$.
\end{thm}

\indent
 From {\rm(h2)} we infer $f(t,x,0)=0$ and $g(t,x,0)=0$, which implies that $(u,v)=(0,0)$ is a solution for the system $(1.1)_ {a, b, c}$. We should point out that if $(u,v)$ satisfies $(1.1)_ {a}$ and $u\not\equiv 0$, then we also have $v\not\equiv 0$ due to the structure of the system $(1.1)_ {a}$. In other words, the problem $(1.1)_{a,b,c}$ possesses no semi-trivial solution of type $(u,0)$ or $(0,v)$.\par

\begin{rmk}\label{rmk1.3}
   By virtue of {\rm{(h1)-(h4)}}, an explicit computation shows some useful facts for proving {\rm Theorem 1.1}{\rm :} \par
  {\rm(i)} $F(t,x,\xi)\geq 0$, and $G(t,x,\xi)\geq 0$ for all $(t,x,\xi)${\rm .}\par
  {\rm(ii)} There are positive numbers $c_1$ and $c_2$, such that
   $$F(t,x,\xi)\geq c_1 |\xi|^{p+1}-c_2,\ \ G(t,x,\xi)\geq c_1 |\xi|^{q+1}-c_2{\rm ,}  \eqno (1.3)
   $$
   and there are constants $\bar{r}$, $c_3 >0$, such that
   $$F(t,x,\xi)\geq c_3 |\xi|^{p+1},\ \ G(t,x,\xi)\geq c_3 |\xi|^{q+1}{\rm ,} \ \  {\rm for}\ |\xi|\geq \bar{r}{\rm ;}
   $$
   \label{} \ \ \ \ furthermore, it follows from {\rm(h3)} and {\rm(h4)} that
   $$\lim\limits_{\xi\rightarrow +\infty} f(t,x,\xi)=\lim\limits_{\xi\rightarrow +\infty} g(t,x,\xi)= +\infty.
   $$
   \label{} \ \ \ \ {\rm(iii)} $F(t,x,\xi)/\xi^2 \rightarrow 0$, and $G(t,x,\xi)/\xi^2 \rightarrow 0$ uniformly in $(t,x)$ as $\xi\rightarrow 0${\rm ;} \qquad\quad\ {\rm(1.4)} \par
   moreover, for each $\nu >0$, there exists a positive number $C_\nu$ such that
   $$|F(t,x,\xi)| \leq \nu \xi^2 + C_\nu |\xi|^{p+1},\ \ {\rm and}\ |G(t,x,\xi)| \leq \nu \xi^2 + C_\nu |\xi|^{q+1}. \eqno (1.5)
   $$
\end{rmk}
   It is well known that such assumptions as (h1)-(h3) are also of great use in solving the nonlinear
elliptic equations and the linearly coupled elliptic systems
\begin{equation*}
  \left\{\begin{array}{l}
     -\Delta u  + u = f(u)+\lambda v ,\;\  x\in {\bf{\mathbb{R}}^N},
    \\[1.2ex]
     -\Delta v  + v = g(v)+\lambda u ,\;\  x\in {\bf{\mathbb{R}}^N},
 \end{array} \right.
  \end{equation*}
  (see \cite{ChZou, ChZou2, Wi} and the references therein). We should point out that, in contrast to the elliptic equations and systems, we may face the following difficulties in the problem of finding periodic solutions for the Klein-Gordon system (1.1)$_{a, b, c}$:\par
  1) The d'Alembert operator $L= \partial_{t}^2 - \partial_{x}^2$ possesses infinitely many eigenvalues  going from $-\infty$ to $+\infty$, so that the positive part and the negative part of the spectrum of $L$ are all {\bf infinite} dimensional spaces, and the functional $\Phi$ corresponding to (1.1)$_ {a, b, c}$ stated in Sect. 2 is neither bounded from above nor
from below. Furthermore, because of the kernel of $L$ is infinite dimensional, the operator $L$ and its inverse are not compact. This fact gives rise to considerable difficulties in solving the strong indefinite problem (1.1)$_ {a, b, c}$. As the lack of compactness properties, the embedding estimates and methods used in \cite{ChZou, ChZou2, Wi} are invalid here.\par
   2) Due to the linear coupling effects, it is hard  to obtain the energy estimate and convexity properties of the
corresponding functional $\Phi$ for the system (1.1)$_ {a, b, c}$. Many  troubles stem from the coupling interplay between the two scalar functions $u$ and $v$. The main challenges in constructing the time-periodic solutions for (1.1)$_ {a, b, c}$ are to control the energy of $\Phi$ in some proper working spaces, and to estimate the asymptotic behavior of the components in the kernel of $L$.\par\medskip
  Since the solutions for (1.1)$_ {a, b, c}$ obtained in Theorem 1.1 are dependent of $\varepsilon$, we are interested in considering the asymptotic behavior of these solutions as $\varepsilon\rightarrow 0$. In the next theorem, we prove:
 \begin{thm}
  Under the conditions of {\rm Theorem 1.1}, and assume $(u_\varepsilon, v_\varepsilon)$ is a solution of ${\rm(1.1)}_{a, b, c}$ obtained in {\rm Theorem 1.1} for $|\varepsilon|< \varepsilon_0$. Let $\varepsilon_n \in (-\varepsilon_0, \varepsilon_0)$ be any sequence with $\varepsilon_n \rightarrow 0$ as $n\rightarrow \infty$. Then, passing to a subsequence, $(u_{\varepsilon_n}, v_{\varepsilon_n})$ converge strongly to $(U_0, V_0)$ in $L^2 (\Omega) \times L^2 (\Omega)$ as $n\rightarrow \infty$, where $U_0$ is a weak solution of {\rm(W1)}, and $V_0$ is a weak solution of {\rm(W2)} respectively.
\end{thm}

   {\bf (b) Regularity results}\par\medskip
   Another problem that we study is the higher regularity of the solutions to the system (1.1)$_ {a, b, c}$.
   We obtain the $L^\infty$ bound of the periodic solutions for (1.1)$_ {a, b, c}$ basing on some precise descriptions of the energy estimates.

\begin{thm}
  Suppose that the conditions of {\rm Theorem 1.1} are satisfied, then there exists $\varepsilon_0 >0$ such that for $|\varepsilon| < \varepsilon_0$, the solution of {\rm(1.1)}$_ {a, b, c}$ lies in $L^\infty (\Omega) \times L^\infty (\Omega)$, where $L^\infty (\Omega)$ is the Lebesgue space equipped with the norm
    $$\|w\|_{L^\infty} = \inf\{C\geq 0:\ |w(t,x)|\leq C\ {\rm for\ almost\ every}\ (t, x)\in \Omega\}<\infty.
    $$
\end{thm}
 Under some more
restrictive assumptions on $f$ and $g$, we have
\begin{thm}
   Under the conditions of {\rm Theorem 1.1}, we assume in addition that $f$,
  $g\in C^1 (\Omega \times \mathbb{R},~\mathbb{R})$,
   and $f(t,x,\xi)$, $g(t,x,\xi)$ are strictly increasing in $\xi$ for all $(t,x)\in \Omega$.\par
   Then for sufficiently small  $\varepsilon$, the solution of ${\rm(1.1)}_{a, b, c}$ is continuous on $\Omega$.
\end{thm}
\indent  Our idea for this theorem is motivated by the works of Rabinowitz \cite{Ra, Ra1} and Br\'{e}zis-Nirenberg \cite{Bre-Nir} which are devoted to the regularity of the solutions for scalar wave equations $w_{tt}-w_{xx}=h(t,x,w)$. The above authors decompose $w$ into a regular term and a null term via the representation theorem (see Lemma 2.13 in \cite{Ra1}). The regular term can be formulated as an integral expression which is continuous in $t$ and $x$. On the other hand, the null term can be controlled by some delicate integral and pointwise estimates, and then the $C^0$-regularity of the solution is guaranteed by a continuity argument.\par
  However, due to the presence of the
linear coupling terms, the previous estimates developed in \cite{Bre-Nir, Ra, Ra1} can not be applied directly to study the higher regularity of the solutions for system (1.1)$_ {a, b, c}$. Especially, the interplay among the linear coupling {\emph {null terms}} is hard to control. To get around this difficulty, a careful calculation and some more precise energy estimates such as (6.3) and (6.4) are required to analyze the interaction between the linear coupling {\emph {null terms}}, provided that $\varepsilon$ is sufficiently small.\par\bigskip

  We conclude this section by illustrating our strategies to tackle the above problems and listing the sketch of the proof of the main results. \par
   First, we prove the existence of the weak solutions for (1.1)$_ {a, b, c}$ by constructing the critical points of the functional $\Phi$ defined by (2.2) restricted in some suitable function spaces, via the local linking method introduced by Shujie Li-Jiaquan Liu \cite{LiLiu, LiuLi}(see also \cite{DL, LiWi, Zhang}). The estimates of the components $(u,v)$ in different parts of the function space $E$ are playing crucial roles in solving this problem.\par
  $\bullet$ For $(u,v)$ that belongs to the orthogonal complement of $\ker  L\times \ker  L$ in $E$, we can control these components by
some compact embedding estimate (2.9); \par
  $\bullet$ For $(u,v)\in \ker L\times \ker L$, the compact properties will be lost. To overcome this difficulty, we apply the
monotonicity technique to analyze the behavior of the nonlinear terms more precisely, as the linear coupling constant $\varepsilon$ is sufficiently small. We require the condition (h4) of Theorem 1.1 to obtain this goal. \par
  Subsequently, we study the asymptotic properties of the solutions constructed in Theorem 1.1 as $\varepsilon\rightarrow 0$. We establish Lemma 4.1 to obtain the uniform bound of the solutions $(u_\varepsilon, v_\varepsilon)\in E$ to (1.1)$_ {a, b, c}$ for any $|\varepsilon|<\varepsilon_0$, which leads to the strong convergence of $(u_\varepsilon, v_\varepsilon) \rightarrow (U_0, V_0)$ in $L^2(\Omega) \times L^2(\Omega)$ as $\varepsilon\rightarrow 0$. Then, we verify that $U_0$, $V_0$ are weak solutions of (W1) and (W2) respectively, by a limiting argument with the aid of some precise energy estimates $(4.9)$-$(4.13)$ in Sect. 4 to approach it. \par
   Finally,  we will improve the regularity of the weak solutions. We carry out the proof by two steps:\par
  S{\scriptsize{TEP}} 1: With the help of the presentation theorem for periodic solutions to the scalar wave equations (\cite{Bre-Cor-Nir, Ra1}), we can use some a-priori estimates and comparison methods to achieve the $L^\infty$-estimate of the solutions $(u,v)$ for (1.1)$_ {a, b, c}$ constructed by Theorem 1.1. The proof depends on the linear coupling structure of the system. See Sect. 5 for details.\par
  S{\scriptsize{TEP}} 2: Relying on the nature of the nonlinearities and the above $L^\infty$-estimate, we can generalize the continuity
method \cite{Bre-Nir, Ra1} used in the scalar wave equations to the case of system (1.1)$_ {a, b, c}$ taking account of the linear coupling effect. For  sufficiently small $\varepsilon$, the integral estimates in \cite{Bre-Nir} are improved  to prove the higher regularity of the solutions (see Sect. 6).  \par\medskip
   We organize the paper as follows. In Sect. 2, we give the functional scheme and define a
suitable function space $E$ to work in it, with the aid of
Fourier expansion formulations for the functions that satisfy (1.1)$_ {b, c}$. Then we introduce a decomposition of $E$ and prepare some basic embedding properties, which enable us to solve (1.1)$_ {a, b, c}$ conveniently. Sect. 3 is devoted to proving Theorem 1.1  via the local linking method. One of the major ingredients in the proof is to verify the (PS)$^*$ condition by a compact argument together with a monotonicity technique (see Lemma 3.1 and Lemma 3.2). Then, in Sect. 4, we investigate the limit behavior of the solutions for (1.1)$_ {a, b, c}$ and prove Theorem 1.3. At last, we turn to study the  further regularity properties of the time-periodic solution for (1.1)$_ {a, b, c}$  and prove Theorem 1.4, Theorem 1.5  in Sect. 5 and Sect. 6 respectively.


\section{The variational framework}
\setcounter{equation}{0}
\label{} \ \ \ \ \ In this section, we present the variational framework which will be used to solve the system (1.1)$_ {a, b, c}$. First, we define the energy functional and its working space as following.

    \subsection{Functional Setting}
  \label{} \ \ \ \   Using the Fourier series, the solutions to the linear equation
  $$\phi_{tt}-\phi_{xx}=h(t,x),\ \ 0<t<2\pi,\ 0<x<\pi,
  $$
   with conditions of $\phi(t,0)=\phi(t,\pi)=0$ and $\phi(t+2\pi,x)=\phi(t,x)$ have a expansion of the form
  $$\phi(t,x)=\sum_{j\in \mathbb{Z_+},~k\in \mathbb{Z}} a_{jk}\sin(jx)\me^{\mi kt},\ \ \ {\rm where}\ \ \overline{a_{jk}}=a_{j, -k}.
  $$
 Then, for $u(t,x)=\sum\limits_{j\in \mathbb{Z_+},~k\in \mathbb{Z}} u_{jk}\sin(jx)\me^{\mi kt}$ and $v=\sum\limits_{j\in \mathbb{Z_+},~k\in \mathbb{Z}}v_{jk}\sin(jx)\me^{\mi kt}$, the inner product in $L^2 (\Omega)$ can be formulated by
   \begin{eqnarray*}
   \lag u, v\rag = \iint_{\Omega}u(t,x)\overline{v(t,x)}\dif t\dif x = \pi ^2 \sum_{j\in \mathbb{Z_+},~k\in \mathbb{Z}} u_{jk}~\overline{v_{jk}},
   \end{eqnarray*}
 and we can write the following quadratic form as
  $$\lag Lu, v\rag = \iint_{\Omega}(\frac{\partial^2 u}{\partial t^2}- \frac{\partial^2 u}{\partial x^2})~\overline{v}~\dif t\dif x
  = \pi ^2 \sum_{j\in \mathbb{Z_+},~k\in \mathbb{Z}} (j^2 - k^2)u_{jk}~\overline{v_{jk}}.
  $$
  \label{} \ \ \ \ \   Motivated by \cite{DL}, it is natural to introduce the Hilbert spaces
  $$\  H=\Big\{u\in L^{2}(\Omega):~\|u\|_{H}^{2}= \pi^2 \sum_{\substack{j\in \mathbb{Z_+},k\in \mathbb{Z}\\j\neq |k|}}
|j^2 -k^2 + b||u_{jk}|^{2} + \pi^2 \sum_{\substack{j\in \mathbb{Z_+},k\in \mathbb{Z}\\j= |k|}}
|u_{jk}|^{2} < \infty  \Big\}, \  \eqno (2.1)
  $$
 and $E=H\times H$ as our working spaces, where $E$ is equipped with the norm $\|(u,v)\|_{E} = (\|u\|_{H}^{2} + \|v\|_{H}^{2})^{1/2}$.\par
   For $(u,v)\in E=H\times H$, let
  \begin{eqnarray*}
 \ \ \ \ \ \ \ \ \ \ \Phi(u,v)&=& -\frac{1}{2}\lag Lu, u\rag - \frac{b}{2}\iint_{\Omega} u^2 \dif t\dif x -\frac{1}{2}\lag Lv, v\rag - \frac{b}{2}\iint_{\Omega} v^2 \dif t\dif x\\
     && -\varepsilon \iint_{\Omega} uv \dif t\dif x -\iint_{\Omega} F(t,x,u) \dif t\dif x -\iint_{\Omega} G(t,x,v) \dif t\dif x. \ \ \ \ \ \ \ (2.2)
  \end{eqnarray*}
  {\bf In the rest of this paper, we denote by $\int \cdot = \iint_{\Omega}\ \cdot\ \dif t\dif x$ for convenience}.\par
   Thus $\Phi$ is a $C^{1}$ functional on $E$, and the  Gateaux derivative of $\Phi$ is
\begin{eqnarray*}\label{3.6}
\ \ \ \lag \Phi'(u,v),~(\varphi, \psi)\rag &=& -\lag Lu,~\varphi\rag - b\int u\varphi -\varepsilon \int  v\varphi -~\lag Lv,~\psi\rag - b\int v\psi -\varepsilon \int  u\psi \\
 && -\int f(t,x,u)\varphi - \int~g(t,x,v)\psi,~\ \ \ \ \forall~(u,v), (\varphi,\psi)\in E. \ \ \ \ (2.3)
 \end{eqnarray*}
 Then $(u,v)$ is a weak solution of system $(1.1)_{a,b,c}$ if and only if $\Phi'(u,v)=0$.

   \subsection{Local linking structure}
  \label{} \ \ \ \   The notion of local linking introduced by S.J. Li and M. Willem \cite{LiWi} is a powerful tool to study the existence of critical points for strongly indefinite
 functionals.\\[0.3em]
   \textit{{\bf Definition 2.1.}}
     Let $E$ be a Banach space, and $E=E^1 \oplus E^2$ is a direct sum decomposition of $E$ (noting that both of $E^1$ and $E^2$ may
 be infinite dimensional spaces). Then $\Phi \in C^1 (E, {\mathbb R})$ is said
 to have a local linking at 0 if for some $r>0$,
 \begin{eqnarray*}
   \Phi(u)\geq 0, && {\rm for}~u\in E^1 ,\ \|u\| \leq r,  \\
   \Phi(u)\leq 0, && {\rm for}~u\in E^2 ,\ \|u\| \leq r.
 \end{eqnarray*}
 \label{} \ \ \   In the case of $\dim E^1 = \dim E^2 =\infty$, it is necessary to explore the Galerkin approximation method and some compactness argument to construct the critical points of the functional $\Phi$ in $E$. To this end, we need the following compactness condition which generalize the (PS) condition.\par\medskip
    Suppose that $E_{1}^{1}\subset E_{2}^{1}\subset \cdots \subset E^1$, $E_{1}^{2}\subset E_{2}^{2}\subset \cdots \subset E^2$ are
 two sequences of finite dimensional subspaces such that
 $$E^j = \overline{\bigcup\limits_{n\in {\mathbb{N}}} E_{n}^j},\ \ \ \ j=1,2.
 $$
  \label{} \ \ \ \ \  For two multi-index $\theta = (\theta^1 , \theta^2)$ and $\beta=(\beta^1, \beta^2)\in {\mathbb {N}}^2$, we denote by $\theta \leq \beta$ if $\theta^1 \leq \beta^1$ and $\theta^2 \leq \beta^2$. A sequence $(\theta_n)\subset {\mathbb {N}}^2$ is said to be admissible if for each $\theta \in {\mathbb {N}}^2$, there is an $m \in {\mathbb {N}}$ such that $\theta_n \geq \theta$ for all $n\geq m$. For $\theta=(\theta^1 , \theta^2)$, let $E_{\theta}= E_{\theta^1}^1 \oplus E_{\theta^2}^2$ and $\Phi_{\theta}=\Phi|_{E_{\theta}}$.\\[0.3em]
   \textit{{\bf Definition 2.2.}}
     The functional $\Phi \in C^1 (E, {\mathbb R})$ is said
 to satisfy the condition (PS)$^*$ if every sequence $(u_{\theta_n}^{\ })$ with $(\theta_n)\subset {\mathbb {N}}^2$ being admissible such that
  $$u_{\theta_n}^{\ } \in E_{\theta_n},\ \ \ \sup\limits_{n}\Phi(u_{\theta_n}^{\ })< \infty\ \ \ {\rm and}\ \ \ \Phi'_{\theta_n}(u_{\theta_n}^{\ })\rightarrow 0 \ \ \ {\rm as}\ \ n\rightarrow \infty
  $$
  contains a subsequence converging to a critical point of $\Phi$.

     We will use the following abstract proposition to solve the system (1.1)$_{a,b,c}$.\\[0.3em]
    \textit{{\bf Proposition A}} (\cite{LiWi, Zhang}). \emph{Suppose that $\Phi\in C^1 (E, {\mathbb {R}})$ and\\
     \indent $(A1)$ $\Phi$ satisfies {\rm(PS)}$^*$ condition{\rm ;}\\
     \indent $(A2)$ $\Phi$ has a local linking at $0;$\\
     \indent $(A3)$ $\Phi$ maps bounded sets into bounded sets$;$\\
     \indent $(A4)$ For every $m\in {\mathbb {N}}$, $\Phi(u)\rightarrow -\infty$ as $\|u\|\rightarrow \infty$, $u\in E_{m}^1 \oplus E^2$.\\
     Then $\Phi$ has a nontrivial critical point $u_0$ in $E$.
    }

    \begin{rmk}
     In {\rm\cite{DL, LiWi}}, it is also pointed out that the critical value corresponding to the critical point $u_0$ obtained in {\rm Proposition A} satisfying
     $$\Phi(u_0)\leq c,\ \ \ where\ c=\sup\limits_{u\in E^1_{m_1 +1} \oplus E^2}\Phi(u), \ and\ m_1\ is\ a\ positive\ integer.
     $$
    \end{rmk}

    In order to apply Proposition A to the functional defined by (2.2), we shall introduce the direct sum decomposition of the
 Banach space $E=H\times H$, where $H$ occurs in (2.1). Some notations are defined as follows: \par
   $\bullet$ $H_b^+$ is the subspace which is spanned by the functions $\sin(jx)\me^{\mi kt}$, where $j\in {\mathbb Z}_+$, $k \in {\mathbb Z}$ satisfying $j^2 -k^2 >-b$ and $j\neq |k|$;\par
   $\bullet$ $H_b^-$ is the subspace which is spanned by the functions $\sin(jx)\me^{\mi kt}$, where $j^2 -k^2 <-b$;\par
   $\bullet$ $H^0 \equiv \ker L$ is the subspace which is spanned by the functions $\sin(jx)\me^{\mi kt}$, for $j= |k|$;\par
   $\bullet$ $E_b^+ = H_b^+ \times H_b^+$, $E_b^- = H_b^- \times H_b^-$, and $E^0 = H^0 \times H^0$;\par
   $\bullet$ $E^1 = E_b^-$, $E^2 = E_b^+ \oplus E^0$, we see $\dim E^j = \infty$, for $j=1$, $2$;\par
   $\bullet$ $E^j_m = span\{e_1^j, \cdots, e_m^j\}$, where $(e_{n}^j)_{n=1}^{\infty}$ is a basis for $E^j$, $j=1$, $2$.\par
   We have $E=E^1 \oplus E^2$. Moreover, $E^1_m$, $E^2_m$ are finite dimensional spaces for every $m\in {\mathbb Z}$, and
 $E_{1}^{1}\subset E_{2}^{1}\subset \cdots \subset E^1$, $E_{1}^{2}\subset E_{2}^{2}\subset \cdots \subset E^2$.\par
  \subsection{Basic estimates}
   \label{} \ \ \ \  At the end of this section,  we list some basic formulas and properties of the Banach spaces on which we
 will work.\par
     For $r\geq 1$, we denote by $L^r (\Omega)$ the space of functions $u(t,x)$ with the norm
     $$\|u\|_{L^r} = \big(\iint_{\Omega}|u(t,x)|^r \dif t
 \dif x \big)^{1/r}.
 $$

    $\bullet$ {\bf Formulas of the norms $\|\cdot\|_{L^2}$, $\|\cdot\|_{H}$ and functional $\Phi$}\par\medskip
    For $u,\ v\in H$, and
    $$u = \sum\limits_{j\in \mathbb{Z_+},~k\in \mathbb{Z}} u_{jk}\sin (jx) \me^{\mi kt},\ \ v = \sum\limits_{j\in \mathbb{Z_+},~k\in \mathbb{Z}} v_{jk}\sin (jx) \me^{\mi kt},
    $$
     let $u = u^+ + u^- + y$, $v = v^+ + v^- + z$, where $u^+,\ v^+ \in H_b^+$, $u^-,\ v^- \in H_b^-$, and
 $y,\ z \in H^0$. By the orthogonality of the subspaces $H_b^+$, $H_b^-$ and $H^0$, we can write the inner product in $L^2(\Omega)$ as $\lag u, v\rag = \lag u^+, v^+\rag + \lag u^-, v^-\rag + \lag y, z\rag$, and $\|u\|_{L^2}^2 = \|u^+\|_{L^2}^2 + \|u^-\|_{L^2}^2 +\|y\|_{L^2}^2 = \pi^2 \sum\limits_{j\in \mathbb{Z_+},~k\in \mathbb{Z}} |u_{jk}|^2$.\par
    Noting that $-b \not\in \sigma(L)$ and $Ly=0$ for $y\in H^0$, we have $\lag Lu, v\rag = \lag Lu^+, v^+\rag + \lag Lu^-, v^-\rag$, and
    $$\lag (L+b)u, u\rag = \pi^2 \sum\limits_{j\in \mathbb{Z_+},~k\in \mathbb{Z}} (j^2 - k^2 + b) |u_{jk}|^2 = \|u^+\|_{H}^2 - \|u^-\|_{H}^2 + b\|y\|_{L^2}^2. \eqno(2.4)
    $$
  \label{}\ \ \ \ \  With the aid of (2.4), we can formulate the energy functional (2.2) as
  \begin{eqnarray*}
   \qquad\qquad \Phi(u,v) &=& -\frac{1}{2}\|u^+\|_{H}^2 + \frac{1}{2} \|u^-\|_{H}^2 - \frac{b}{2} \|y\|_{L^2}^2
      -\frac{1}{2}\|v^+\|_{H}^2 + \frac{1}{2} \|v^-\|_{H}^2 - \frac{b}{2} \|z\|_{L^2}^2\\
      &&-\ \varepsilon \int uv - \int F(t,x,u) - \int G(t,x,v). \qquad\qquad\qquad\qquad\quad\ \ (2.5)
  \end{eqnarray*}
 \label{}\ \ \ \ \   $\bullet$ {\bf Some embedding estimates}\par\medskip
   In view of (1.2), there is $\eta >0$, such that
   $$\|u^+\|_H ^2 = \pi^2 \sum\limits_{\substack{j^2 - k^2 > -b\\j\neq |k|}}
(j^2 -k^2 + b)|u_{jk}|^{2} \geq \eta \|u^+\|_{L^2} ^2,\ {\rm and}\ \|u^-\|_H ^2 \geq \eta \|u^-\|_{L^2} ^2. \eqno (2.6)
$$
 Thus, by (2.1) and (2.6) we have
$$\|u\|_{L^2} ^2 \leq  \kappa\|u\|_{H} ^2,\ {\rm where}\ \kappa= \max\{ 1/\eta,\ 1\}.  \eqno (2.7)
$$
 \label{}\ \ \ \ \   The following properties are well known (see \cite{DL} for instance):
   $$\|w\|_{L^r} \leq C \|w\|_{H},\ \ \ {\rm for}\ w\in H_b^+ \oplus H_b^- \ {\rm and}\ r\geq 1, \eqno (2.8)
   $$
    where $C>0$ is a constant which only depends on $r$.
Furthermore, the embedding
$$ H_b^+ \oplus H_b^- \hookrightarrow L^r (\Omega)\ {\rm is\ compact,\ \ for}\ r\geq 1.  \eqno (2.9)
$$
\begin{rmk}
 Let us point out that we lose the compact embedding from $H^0$ to $L^r (\Omega)$ for $r>2$, because of $0\in \sigma(L)$ and $\dim \ker L =\infty$. Moreover, according to the definition of $H$ (see {\rm(2.1)}), we do not obtain that $\|w\|_{L^r} \leq C \|w\|_H$, for any $w\in H$, and $r>2$. Thus, we need careful computations to study the behavior of the null components and nonlinear terms appearing in the functional $\Phi$.
\end{rmk}

\section{Existence of weak solutions}
\setcounter{equation}{0}
\label{}\ \ \ \
  In this section, we prove Theorem 1.1. To achieve this goal, we will check that the conditions of Proposition A hold for the functional $\Phi$ defined by (2.2). We use $C$ and $c_{*}$, $d_*$ with quantity subscripts to stand for different constants in the rest of this article.\par
   {\bf {Verification of \ {\rm (A1)}}:}\\
\label{}\ \ \ \ Let $(u_{\theta_n}^{\ }, v_{\theta_n}^{\ }) \in E_{\theta_n}:= E^1_{\theta^1 _n} \oplus E^2_{\theta^2 _n}$,
 where $\theta_n = (\theta_n^1 , \theta_n^2) \in {\mathbb N}^2$, the sequence $\{\theta_n\}_{n=1}^\infty$ is admissible, and $E^1_{\theta^1 _n}$, $E^2_{\theta^2 _n}$ are
 finite dimensional spaces defined in subsection 2.2.

   We suppose $\{(u^{\ }_{\theta_n}, v^{\ }_{\theta_n})\}_{n=1}^\infty$ is a (PS)$^*$ sequence of $\Phi$, that is $d:= \sup\limits_{n}\Phi(u^{\ }_{\theta_n}, v^{\ }_{\theta_n})< \infty$, and $\Phi'_{\theta_n}(u^{\ }_{\theta_n}, v^{\ }_{\theta_n})\rightarrow 0$ as $n \rightarrow \infty$, where $\Phi'_{\theta_n}= \big(\Phi|_{E_{\theta_n}}\big)'$. First, under the assumptions of Theorem 1.1, we show: \\[0.3em]
 \textit{{\bf Lemma 3.1.}} \emph{Any} {\rm (PS)}$^{*}$ \emph{sequence is bounded.}
\begin{proof}
  For simplicity, we denote by $(u, v) = (u_{\theta_n}^{\ }, v_{\theta_n}^{\ })$.\par
\emph{Step 1. Estimates of $\|u\|_{L^{p+1}}$ and $\|v\|_{L^{q+1}}$}\par
   By assumption (h3) and (1.3), a direct calculation shows
   \begin{eqnarray*}
     &&\Phi(u,v)-\frac{1}{2}\lag \Phi'_{\theta_n}(u,v),\ (u,v)\rag\\
      &=& \frac{1}{2}\int f(t,x,u)u - \int F(t,x,u) + \frac{1}{2}\int g(t,x,v)v - \int G(t,x,v)\\
      &\geq& \frac{p-1}{2} \int F(t,x,u) + \frac{q-1}{2} \int G(t,x,v) \ \geq \ c_1 \|u\|_{L^{p+1}}^{p+1} + c_1 \|v\|_{L^{q+1}}^{q+1} -c_2.
   \end{eqnarray*}
   Since $\{(u,v)\} = \{(u^{\ }_{\theta_n}, v^{\ }_{\theta_n})\}$ is a (PS)$^*$ sequence of $\Phi$, for $n$ is large enough, we have
  $$c_1 \|u\|_{L^{p+1}}^{p+1} + c_1 \|v\|_{L^{q+1}}^{q+1} -c_2 \leq d + \|(u,v)\|_E.
  $$
  Hence, there is $c_3 >0$, such that
  $$\|u\|_{L^{p+1}} \leq c_3 + c_3 \|(u,v)\|_{E}^{\frac{1}{p+1}},\ \ {\rm and}\ \|v\|_{L^{q+1}} \leq c_3 + c_3 \|(u,v)\|_{E}^{\frac{1}{q+1}}. \eqno (3.1)
  $$
\label{}\ \ \ \ \emph{Step 2. Estimates of $\|u^{\pm}\|_{H}$, $\|v^{\pm}\|_{H}$, $\|y\|_{L^2}$ and $\|z\|_{L^2}$}\par
   We decompose $u= u^+ +u^- +y$, $v=v^+ +v^- +z$, where $u^+$, $v^+ \in H_{b}^+$, $u^-$, $v^- \in H_{b}^-$, and $y$, $z\in H^0$.
Noting that $\|u\|_{L^2}^2 = \|u^+\|_{L^2}^2 + \|u^-\|_{L^2}^2 + \|y\|_{L^2}^2$, then by virtue of $p$, $q>1$ and (3.1), we estimate
 $$\|y\|_{L^2}^2 \leq \|u\|_{L^2}^2 \leq c\|u\|_{L^{p+1}}^2 \leq c_4 + c_4 \|(u,v)\|_{E}^{\frac{2}{p+1}}, \eqno (3.2)
 $$
 and
 $$\|z\|_{L^2}^2 \leq \|v\|_{L^2}^2 \leq c\|v\|_{L^{q+1}}^2 \leq c_4 + c_4 \|(u,v)\|_{E}^{\frac{2}{q+1}}. \eqno (3.3)
 $$
 \label{}\ \ \ \ Taking $(\varphi, \psi) = (u^+, v^+)$ in (2.3), from the orthogonality of the subspaces $H_b^+$, $H_b^-$ and $H^0$, we get
 \begin{eqnarray*}
 \lag \Phi_{\theta_n}'(u,v),~(u^+, v^+)\rag &=& - \lag (L+b)u^+, u^+\rag - \lag (L+b)v^+, v^+\rag - 2\varepsilon \int u^+ v^+ \\
 &&- \int f(t,x,u)u^+ - \int g(t,x,v)v^+.
 \end{eqnarray*}
 In fact of $\|u^+\|_H^2 = \lag (L+b)u^+, u^+\rag$, and $\|v^+\|_H^2 = \lag (L+b)v^+, v^+\rag$ for $u^+$, $v^+ \in H_b ^+$, when $n$ is large enough we obtain
 $$\|u^+\|_H ^2 + \|v^+\|_H ^2 \leq o(1) - 2\varepsilon \int u^+ v^+ - \int f(t,x,u)u^+ - \int g(t,x,v)v^+.  \eqno (3.4)
 $$
 \label{}\ \ \ \ A similar argument as in (3.2) provides that
 \begin{eqnarray*}
\ \ \ \ \ \qquad - 2\varepsilon \int u^+ v^+ &\leq& 2\;|\varepsilon|\;\|u^+\|_{L^2}\|v^+\|_{L^2} \leq c \|u\|_{L^{p+1}}\|v\|_{L^{q+1}}\\
 &\leq& c_5 + c_5 \|(u,v)\|_E ^{\frac{1}{p+1}} + c_5 \|(u,v)\|_E ^{\frac{1}{q+1}} + c_5 \|(u,v)\|_E ^{\frac{1}{p+1}+\frac{1}{q+1}}.\ \ \ \ (3.5)
 \end{eqnarray*}
\label{}\ \ \ \ By assumption (h1), (3.1) and the H\"{o}lder inequality, we get
  \begin{eqnarray*}
  \qquad\qquad - \int f(t,x,u)u^+ &\leq& c_0 \int \big(|u^+| + |u|^p |u^+|\big)\ \leq \ c \|u^+\|_{L^2} + c_0 \|u\|_{L^{p+1}}^{p} \|u^+\|_{L^{p+1}}\\
    &\leq& c_6 \big(1+ \|(u,v)\|_{E}^{\frac{p}{p+1}}\big)\|u^+\|_H, \quad\qquad\qquad\qquad\qquad\qquad (3.6)
  \end{eqnarray*}
  where the last inequality is deduced by (2.7), (2.8) and (3.1). Similarly, we have
    $$- \int g(t,x,v)v^+ \leq c \|v^+\|_{L^2} + c_0 \|v\|_{L^{q+1}}^{q} \|v^+\|_{L^{q+1}} \leq c_6 \big(1+ \|(u,v)\|_{E}^{\frac{q}{q+1}}\big)\|v^+\|_H.    \eqno (3.7)
    $$
  \label{} \ \ \ \ Inserting (3.5)-(3.7) into the right hand side of (3.4), and by the inequalities that $\|u^+\|_H \leq \|(u,v)\|_E$ and $\|v^+\|_H \leq \|(u,v)\|_E$, we know
  \begin{eqnarray*}
  \qquad \|u^+\|_H ^2 + \|v^+\|_H ^2 &\leq& o(1)+ c_5 + c_5 \|(u,v)\|_E ^{\frac{1}{p+1}} + c_5 \|(u,v)\|_E ^{\frac{1}{q+1}} + c_5 \|(u,v)\|_E ^{\frac{1}{p+1}+\frac{1}{q+1}}\\
  && +\ c_6 \big(1+ \|(u,v)\|_{E}^{\frac{p}{p+1}}+\|(u,v)\|_{E}^{\frac{q}{q+1}}\big)\|(u,v)\|_E.\qquad \qquad\qquad (3.8)
 \end{eqnarray*}
 \label{}\ \ \ \ For $u^-$, $v^- \in H_b ^-$, analogue to (3.8) we also derive
 \begin{eqnarray*}
  \quad \|u^-\|_H ^2 + \|v^-\|_H ^2 &=& \lag \Phi_{\theta_n}'(u,v),~(u^-, v^-)\rag +2\varepsilon \int u^- v^- + \int f(t,x,u)u^- + \int g(t,x,v)v^- \\
  &\leq& c'_5 + c'_5 \|(u,v)\|_E ^{\frac{1}{p+1}} + c'_5 \|(u,v)\|_E ^{\frac{1}{q+1}} + c'_5 \|(u,v)\|_E ^{\frac{1}{p+1}+\frac{1}{q+1}}\\[0.4em]
  && +\ c'_6 \big(1+ \|(u,v)\|_{E}^{\frac{p}{p+1}}+\|(u,v)\|_{E}^{\frac{q}{q+1}}\big)\|(u,v)\|_E. \quad\qquad\qquad\qquad (3.9)
 \end{eqnarray*}
\label{}\ \ \ \ \emph{Step 3. Bound of} $\|(u,v)\|_{E}$\par
  Observing that
  $$\|(u,v)\|_{E}^2 = \|u^+\|_H ^2 + \|v^+\|_H ^2 + \|u^-\|_H ^2 + \|v^-\|_H ^2 + \|y\|_{L^2} ^2 + \|z\|_{L^2} ^2,
  $$
and using the estimates of (3.2), (3.3), (3.8), (3.9), we arrive at
  \begin{eqnarray*}
    \|(u,v)\|_{E}^2 &\leq& c_7 + c_7 \|(u,v)\|_E ^{\frac{1}{p+1}}+ c_7 \|(u,v)\|_E ^{\frac{1}{q+1}} + c_7 \|(u,v)\|_E ^{\frac{1}{p+1}+\frac{1}{q+1}} + c_7 \|(u,v)\|_E ^{\frac{2}{p+1}}\\
    && +\ c_7 \|(u,v)\|_E ^{\frac{2}{q+1}} + c_7 \|(u,v)\|_E
   + c_7 \|(u,v)\|_E ^{\frac{p}{p+1}+1} + c_7 \|(u,v)\|_E ^{\frac{q}{q+1}+1}.
  \end{eqnarray*}
 \label{}\ \ \ \ In view of $p$, $q>1$, we find all the powers of $\|(u,v)\|_E$ in the right hand side of the preceding inequality are less than 2. Hence, there exists $M>0$ which is independent of $n$, such that $\|(u,v)\|_E \leq M$, and we conclude that any (PS)$^*$ sequence $\{(u^{\ }_{\theta_n}, v^{\ }_{\theta_n})\}$ of $\Phi$ is bounded.
\end{proof}
  In what follows, under the assumptions of Theorem 1.1, we assert that the functional $\Phi$ satisfies the (PS)$^*$ condition. \\[0.3em]
 \textit{{\bf Lemma 3.2.}} \emph{Let $\{(u_{\theta_n}^{\ }, v_{\theta_n}^{\ })\}$ be a {\rm (PS)}$^{*}$ sequence of $\Phi$, then $\{(u_{\theta_n}^{\ }, v_{\theta_n}^{\ })\}$ contains a subsequence which converges to a critical point of $\Phi$.}
\begin{proof}
Since $E$ is a Hilbert space, then Lemma 3.1 guarantees that $\{(u_{\theta_n}^{\ }, v_{\theta_n}^{\ })\}$ converge weakly to some $(u,v)\in E$ along with a subsequence. Decomposing $u_{\theta_n}^{\ } = u^+_{\theta_n} + u^-_{\theta_n} + y^{\ }_{\theta_n}$, $v_{\theta_n}^{\ } = v^+_{\theta_n} + v^-_{\theta_n} + z^{\ }_{\theta_n}$, where $(u^+_{\theta_n}, v^+_{\theta_n}) \in E^+_b$, $(u^-_{\theta_n}, v^-_{\theta_n}) \in E^-_b$, and $(y^{\ }_{\theta_n}, z^{\ }_{\theta_n}) \in E^0$. Let $(u^+, v^+) \in E^+_b$, $(u^-, v^-) \in E^-_b$ and $(y, z) \in E^0$ are the weak limits of $\{(u^+_{\theta_n}, v^+_{\theta_n})\}$, $\{(u^-_{\theta_n}, v^-_{\theta_n})\}$ and $\{(y_{\theta_n}^{\ }, z_{\theta_n}^{\ })\}$ respectively.\par
  We will conclude that $\{(u_{\theta_n}^{\ }, v_{\theta_n}^{\ })\}$ converge {\emph{strongly}} to $(u, v) = (u^+ + u^- +y,\ v^+ + v^- +z)$, by extracting a subsequence if necessary.\par\medskip
  \emph{Strong convergence of $\{(u^+_{\theta_n}, v^+_{\theta_n})\}$ and $\{(u^-_{\theta_n}, v^-_{\theta_n})\}$ in $E$}\par\medskip
  For $(u^+_{\theta_n}, v^+_{\theta_n})\in E_b^+$ and $(u^+, v^+)\in E_b^+$, by virtue of the weak convergence of $u^+_{\theta_n} \rightharpoonup u^+$ and $v^+_{\theta_n} \rightharpoonup v^+$ in $H$, we have $\lag (L+b)u^+, u^+_{\theta_n} - u^+ \rag \rightarrow 0$ and $\lag (L+b)v^+, v^+_{\theta_n} - v^+ \rag  \rightarrow 0$ as $n\rightarrow \infty$.
  Hence it follows that, for large $n$,
  \begin{eqnarray*}
    \qquad\quad\qquad &&\|u^+_{\theta_n} - u^+\|_H^2 + \|v^+_{\theta_n} - v^+\|_H^2\\[0.3em]
     &=& \lag (L+b)(u^+_{\theta_n} - u^+), u^+_{\theta_n} - u^+ \rag +
     \lag (L+b)(v^+_{\theta_n} - v^+), v^+_{\theta_n} - v^+ \rag \\[0.3em]
     &=& \lag (L+b)u^+_{\theta_n}, u^+_{\theta_n} - u^+ \rag +\lag (L+b)v^+_{\theta_n}, v^+_{\theta_n} - v^+\rag +o(1).   \quad\qquad\qquad   (3.10)
  \end{eqnarray*}
  From (2.3), the first two terms in the right hand side of (3.10) can be expressed by
  \begin{eqnarray*}
    \qquad &&\lag (L+b)u^+_{\theta_n}, u^+_{\theta_n} - u^+ \rag +\lag (L+b)v^+_{\theta_n}, v^+_{\theta_n} - v^+\rag\\[0.4em]
   &=& -\ \lag \Phi'(u_{\theta_n}^{\ }, v_{\theta_n}^{\ }), (u^+_{\theta_n} - u^+, v^+_{\theta_n} - v^+)\rag   - \int f(t,x,u_{\theta_n}^{\ })(u^+_{\theta_n} - u^+) \\
   && -\ \varepsilon \int v_{\theta_n}^{\ }(u^+_{\theta_n} - u^+) - \varepsilon \int u_{\theta_n}^{\ }(v^+_{\theta_n} - v^+) - \int g(t,x,v_{\theta_n}^{\ })(v^+_{\theta_n} - v^+).\qquad (3.11)
  \end{eqnarray*}
\label{}\ \ \ \    To control (3.11), we denote by $P_{\theta_n}$ the projection operator from $E$ to its subspace $E_{\theta_n}$, and  we represent the first term in the right hand side of (3.11) as
\begin{eqnarray*}
  \lag \Phi'(u_{\theta_n}^{\ }, v_{\theta_n}^{\ }), (u^+_{\theta_n} - u^+, v^+_{\theta_n} - v^+)\rag
  &=& \lag \Phi_{\theta_n}'(u_{\theta_n}^{\ }, v_{\theta_n}^{\ }), (u^+_{\theta_n} - P_{\theta_n}u^+, v^+_{\theta_n} - P_{\theta_n}v^+)\rag\\[0.4em]
  && \ -\ \lag \Phi'(u_{\theta_n}^{\ }, v_{\theta_n}^{\ }), (I-P_{\theta_n}) (u^+, v^+)\rag.
 \end{eqnarray*}
 Noting that $(u^+_{\theta_n} - P_{\theta_n}u^+,\ v^+_{\theta_n} - P_{\theta_n}v^+) \in E_{\theta_n}$, then the facts of $\Phi_{\theta_n}'(u_{\theta_n}^{\ }, v_{\theta_n}^{\ }) \rightarrow 0$ and $(I-P_{\theta_n}) (u^+, v^+) \rightarrow 0$ assure that
 $$\lag \Phi'(u_{\theta_n}^{\ }, v_{\theta_n}^{\ }), (u^+_{\theta_n} - u^+, v^+_{\theta_n} - v^+)\rag \rightarrow 0,\ \ {\rm as}\ n\rightarrow \infty.     \eqno (3.12)
 $$
\label{}\ \ \ \ Observe that $H_b^+ \oplus H_b^-$ embeds compactly into $L^r (\Omega)$ for $r\geq 1$, then $u^+_{\theta_n} \rightharpoonup u^+$, $v^+_{\theta_n} \rightharpoonup v^+$ in $H$ imply that $u^+_{\theta_n} \rightarrow u^+$, $v^+_{\theta_n} \rightarrow v^+$ strongly in $L^r (\Omega)$ for every $r\geq 1$. Moreover, since $\{u^{\ }_{\theta_n}\}$, $\{v^{\ }_{\theta_n}\}$ are bounded in $H$, and from (2.7) we obtain
\begin{eqnarray*}
  \qquad\qquad\left |\varepsilon \int v^{\ }_{\theta_n} (u^+_{\theta_n} - u^+)  \right |
  &\leq& |\varepsilon|\; \|v^{\ }_{\theta_n}\|_{L^2} \|u^+_{\theta_n} - u^+\|_{L^2} \rightarrow 0, \qquad\qquad\qquad\qquad\  (3.13)\\
  \left |\varepsilon \int u^{\ }_{\theta_n} (v^+_{\theta_n} - v^+)  \right |
  &\leq& |\varepsilon|\; \|u^{\ }_{\theta_n}\|_{L^2} \|v^+_{\theta_n} - v^+\|_{L^2} \rightarrow 0,\ \ {\rm as}\ n\to \infty.\quad\qquad (3.14)
\end{eqnarray*}
By the assumption (h1) together with H\"{o}lder inequality, it follows from (3.1) and the boundedness of $\|(u^{\ }_{\theta_n}, v^{\ }_{\theta_n})\|_E$ that
\begin{eqnarray*}
  \left |\int f(t,x,u^{\ }_{\theta_n}) (u^+_{\theta_n} - u^+)  \right |
  &\leq& c_8\|u^+_{\theta_n} - u^+\|_{L^2} + c_8 \|u^{\ }_{\theta_n}\|^{p}_{L^{p+1}} \|u^+_{\theta_n} - u^+\|_{L^{p+1}}\rightarrow 0,\ \ (3.15)\\
  \left |\int g(t,x,v^{\ }_{\theta_n}) (v^+_{\theta_n} - v^+)  \right |
  &\leq& c_8\|v^+_{\theta_n} - v^+\|_{L^2} + c_8 \|v^{\ }_{\theta_n}\|^{q}_{L^{q+1}} \|v^+_{\theta_n} - v^+\|_{L^{q+1}}\rightarrow 0,\ \ \ (3.16)
\end{eqnarray*}
as $n\to \infty$.\par
   Collecting (3.10)-(3.16), we infer
$$\|u^+_{\theta_n} - u^+\|_H^2 + \|v^+_{\theta_n} - v^+\|_H^2 \rightarrow 0, \ \ {\rm as}\ n\rightarrow \infty,
$$
 which means $\{(u^+_{\theta_n}, v^+_{\theta_n})\}$ converge to $(u^+, v^+)$ strongly in $E$. Furthermore, for $(u^-_{\theta_n}, v^-_{\theta_n})\in E_b^-$ and $(u^-, v^-)\in E_b^-$, a similar argument allows us to obtain $(u^-_{\theta_n}, v^-_{\theta_n}) \rightarrow (u^-, v^-)$ strongly in $E$.\par
   To finish the proof of this lemma, it remains to derive the strong convergence of $(y^{\ }_{\theta_n}, z^{\ }_{\theta_n}) \rightarrow (y,
 z)$ in $E$, where $(y^{\ }_{\theta_n}, z^{\ }_{\theta_n})$, $(y, z) \in E^0$. Keeping in mind that the embedding from $E^0$ to $L^r (\Omega)$ is not compact, thus the above procedure is invalid to control the components $(y,z) \in E^0$. In order to deal with the difficulties stem from the lack of compactness, we shall employ the monotonicity method to study the asymptotic behavior of $\{(y^{\ }_{\theta_n}, z^{\ }_{\theta_n})\}$.\par\medskip
  \emph{Strong convergence of $\{(y^{\ }_{\theta_n}, z^{\ }_{\theta_n})\}$ in $E$ for small $\varepsilon$}\par\medskip
  Recall (2.1), it suffices to show that $\|y^{\ }_{\theta_n} - y\|_{L^2} \rightarrow 0$,  $\|z^{\ }_{\theta_n} - z\|_{L^2} \rightarrow 0$ as
$n\rightarrow \infty$, where $(y, z)$ is the weak limit of $\{(y^{\ }_{\theta_n}, z^{\ }_{\theta_n})\}$ in $E^0$. \par
  From (2.3) and set $(\varphi, \psi)= (y^{\ }_{\theta_n}-y, z^{\ }_{\theta_n}-z)$, then
  $$\lag Lu^{\ }_{\theta_n},  y^{\ }_{\theta_n}-y\rag = \lag Ly^{\ }_{\theta_n},  y^{\ }_{\theta_n}-y\rag =0,\ \ \lag Lv^{\ }_{\theta_n},  z^{\ }_{\theta_n}-z\rag = \lag Lz^{\ }_{\theta_n},  z^{\ }_{\theta_n}-z\rag =0
  $$
  imply that
  \begin{eqnarray*}
   &&-\lag \Phi'(u^{\ }_{\theta_n}, v^{\ }_{\theta_n}), (y^{\ }_{\theta_n} - y, z^{\ }_{\theta_n} - z)\rag = \int f(t,x,u^{\ }_{\theta_n})(y^{\ }_{\theta_n} - y) + \int g(t,x,v^{\ }_{\theta_n})(z^{\ }_{\theta_n} - z)  \\
    &+& b\int u^{\ }_{\theta_n} (y^{\ }_{\theta_n}-y)
    + \varepsilon \int v^{\ }_{\theta_n} (y^{\ }_{\theta_n}-y) + b\int v^{\ }_{\theta_n} (z^{\ }_{\theta_n}-z) + \varepsilon \int u^{\ }_{\theta_n} (z^{\ }_{\theta_n}-z). \qquad (3.17)
  \end{eqnarray*}
\label{}\ \ \ \ Since $y^{\ }_{\theta_n} \rightharpoonup y$, $z^{\ }_{\theta_n}\rightharpoonup z$ in $L^2(\Omega)$, and by the orthogonality of the subspaces $H_b^+$, $H_b^-$ and $H^0$, we conclude that
 \begin{eqnarray*}
   b\int u^{\ }_{\theta_n} (y^{\ }_{\theta_n}-y) + b\int v^{\ }_{\theta_n} (z^{\ }_{\theta_n}-z) &=& b\int y^{\ }_{\theta_n} (y^{\ }_{\theta_n}-y) + b\int z^{\ }_{\theta_n} (z^{\ }_{\theta_n}-z)\\
   &=& b\int (y^{\ }_{\theta_n}-y)^2 + b\int (z^{\ }_{\theta_n}-z)^2 + o(1),
 \end{eqnarray*}
 and
  \begin{eqnarray*}
    \varepsilon \int v^{\ }_{\theta_n} (y^{\ }_{\theta_n}-y) + \varepsilon \int u^{\ }_{\theta_n} (z^{\ }_{\theta_n}-z) &=& \varepsilon \int z^{\ }_{\theta_n} (y^{\ }_{\theta_n}-y) + \varepsilon \int y^{\ }_{\theta_n} (z^{\ }_{\theta_n}-z)\\
    &=& 2 \varepsilon \int (y^{\ }_{\theta_n}-y) (z^{\ }_{\theta_n}-z) +o(1),\ \ \ {\rm as}\ n\rightarrow \infty.
  \end{eqnarray*}
 Moreover, an analogue of argument in (3.12) gives $\lag \Phi'(u^{\ }_{\theta_n}, v^{\ }_{\theta_n}), (y^{\ }_{\theta_n} - y, z^{\ }_{\theta_n} - z)\rag \rightarrow 0$, as $n\to \infty$.\par
   Therefore, the previous three estimates and (3.17) allow us to deduce that
   \begin{eqnarray*}
  \qquad  && (b-|\varepsilon|)\|y^{\ }_{\theta_n}-y\|^2_{L^2} + (b-|\varepsilon|)\|z^{\ }_{\theta_n}-z\|^2_{L^2} \\[0.4em]
   &\leq& b\|y^{\ }_{\theta_n}-y\|^2_{L^2} + b\|z^{\ }_{\theta_n}-z\|^2_{L^2} + 2 \varepsilon \int (y^{\ }_{\theta_n}-y) (z^{\ }_{\theta_n}-z) \\
   &=& -\int f(t,x,u^{\ }_{\theta_n})(y^{\ }_{\theta_n} - y) - \int g(t,x,v^{\ }_{\theta_n})(z^{\ }_{\theta_n} - z) +o(1),\ {\rm as}\ n\to \infty.\qquad (3.18)
  \end{eqnarray*}
 \label{}\ \ \ \ To control the right hand side of (3.18), we rewrite
  \begin{eqnarray*}
   \quad \int f(t,x,u^{\ }_{\theta_n})(y^{\ }_{\theta_n} - y) &=& \int [f(t,x,u^+_{\theta_n}+u^-_{\theta_n}+y^{\ }_{\theta_n})
   - f(t,x,u^+_{\theta_n}+u^-_{\theta_n}+y)](y^{\ }_{\theta_n} - y)\\
   &+& \int [f(t,x,u^+_{\theta_n}+u^-_{\theta_n}+y)
   - f(t,x,u^+ + u^- +y)](y^{\ }_{\theta_n} - y)\\
    &+& \int f(t,x,u^+ + u^- +y)(y^{\ }_{\theta_n} - y) := I_1 + I_2 + I_3.  \qquad\ \  (3.19)
  \end{eqnarray*}
  \label{}\ \ \ \ We estimate $I_1$, $I_2$ and $I_3$ in the sequel.\par
   Noting that $f(t,x,\xi)$ is nondecreasing in $\xi$ by the condition (h4) of Theorem 1.1, we have $I_1 \geq 0$ immediately.\par
  To estimate $I_2$ and $I_3$, firstly we check $y^{\ }_{\theta_n} \rightharpoonup y$ weakly in $L^{p+1}(\Omega)$. For $u^{\ }_{\theta_n} = u^+_{\theta_n} + u^-_{\theta_n} + y^{\ }_{\theta_n}$, in view of (3.1) and the embedding $H^{\pm}_b \hookrightarrow L^{p+1}(\Omega)$, we get
   \begin{eqnarray*}
    \qquad \|y^{\ }_{\theta_n}\|_{L^{p+1}} &\leq& \|u^{\ }_{\theta_n}\|_{L^{p+1}} + \|u^+_{\theta_n}\|_{L^{p+1}} + \|u^-_{\theta_n}\|_{L^{p+1}}
     \leq \|u^{\ }_{\theta_n}\|_{L^{p+1}} + c\|u^+_{\theta_n}\|_H + c\|u^-_{\theta_n}\|_H\\
      &\leq& c_3 + c_3 \|(u^{\ }_{\theta_n}, v^{\ }_{\theta_n})\|_{E}^{\frac{1}{p+1}} + 2c \|(u^{\ }_{\theta_n}, v^{\ }_{\theta_n})\|_{E}.  \qquad\qquad\qquad\qquad\qquad\  (3.20)
   \end{eqnarray*}
 Then the boundedness of the (PS)$^*$ sequence $\{(u^{\ }_{\theta_n}, v^{\ }_{\theta_n})\}$ ensures that $\{y^{\ }_{\theta_n}\}$ is bounded in $L^{p+1}(\Omega)$, and $\{y^{\ }_{\theta_n}\}$ possesses a subsequence which converge weakly in $L^{p+1}(\Omega)$. Recording that $L^{p+1}(\Omega) \hookrightarrow L^2(\Omega)$ for $p>1$ and $y$ is the weak limit of $\{y^{\ }_{\theta_n}\}$ in $L^2(\Omega)$, hence by the uniqueness of weak limit, we have $y^{\ }_{\theta_n} \rightharpoonup y$ weakly in $L^{p+1}(\Omega)$, with a subsequence still renamed by $\{y^{\ }_{\theta_n}\}$.\par
    By the condition (h1) of Theorem 1.1, we know the operator $f: \xi \mapsto f(t,x,\xi)$ is continuous from $L^{p+1}(\Omega)$ to
 $L^{\frac{p+1}{p}}(\Omega)$. Since $u^+_{\theta_n}\rightarrow u^+$ and $u^-_{\theta_n}\rightarrow u^-$ in $H$, we have  $u^+_{\theta_n} + u^-_{\theta_n} \rightarrow u^+ + u^-$ strongly in $L^{p+1}(\Omega)$ via the embedding $H^+_b\oplus H^-_b \hookrightarrow L^{p+1}(\Omega)$. Therefore,
  $$f(t,x,u^+_{\theta_n}+u^-_{\theta_n}+y)
   - f(t,x,u^+ + u^- +y) \rightarrow 0\ {\rm in}\ L^{\frac{p+1}{p}}_{\ },\ \ {\rm as}\ n\rightarrow \infty,
  $$
 and $f(t,x,u^+ + u^- +y)$ is in $L^{\frac{p+1}{p}}(\Omega)$. By virtue of $y^{\ }_{\theta_n}-y \in L^{p+1}$ and $L^{\frac{p+1}{p}}(\Omega)=(L^{p+1}(\Omega))^*$, we obtain that $I_2 \rightarrow 0$ and $I_3 \rightarrow 0$, as $n\rightarrow \infty$.\par
   With the estimates of $I_1$, $I_2$, $I_3$ in hand, then passing the limit in (3.19) yields that
   $$\int f(t,x,u^{\ }_{\theta_n})(y^{\ }_{\theta_n} - y) \geq 0,\ \ {\rm as}\ n\rightarrow \infty. \eqno(3.21)
   $$
    Moreover, with a similar computation, we find
 $$\int g(t,x,v^{\ }_{\theta_n})(z^{\ }_{\theta_n} - z) \geq 0,\ \  {\rm as}\ n\rightarrow \infty.
  $$
  Hence, turning back to (3.18) and choosing $|\varepsilon| < b$, we have $\|y^{\ }_{\theta_n}-y\|^2_{L^2} + \|z^{\ }_{\theta_n}-z\|^2_{L^2} \leq o(1)$ as $n\rightarrow \infty$. Thereby the proof of Lemma 3.2 is completed.
\end{proof}
   {\bf {Proof of \ {\rm (A2)}}:}\\
\label{}\ \ \ \ The following lemma shows the functional $\Phi$ satisfying the local linking structure for small $\varepsilon$.\\[0.3em]
 \textit{{\bf Lemma 3.3.}} \emph{Assume that} (h1)$-$(h4) \emph{of Theorem} 1.1 \emph{are satisfied, then for $\varepsilon$ is sufficiently small, there exists $\rho >0$, such that}\par
  (i) $\Phi(u,v)\geq 0, \ for\ (u,v)\in E^1\cap B_{\rho}$, \par
   (ii) $\Phi(u,v)\leq 0, \ for\ (u,v)\in E^2\cap B_{\rho}$,\\
\emph{where} $E^1 = E^-_{b}$, $E^2 = E^+_{b}\oplus E^0$, and $B_{\rho} = \left\{(u,v)\big| \|(u,v)\|_{E} \leq \rho\right\}$.
\begin{proof}
  (i) For $(u,v)\in E^-_b = H^-_b \times H^-_b$, by (2.4) we have
  $$\|(u,v)\|_E^2 = \|u\|_H^2 + \|v\|_H^2 = - \lag(L+b)u, u\rag - \lag(L+b)v, v\rag.
   $$
   Thus the energy functional (2.5) can be represented by
   $$ \Phi(u,v)= \frac{1}{2}\|(u,v)\|_{E}^2 -\varepsilon \int uv  -\int F(t,x,u) -\int G(t,x,v).
   $$
   \label{}\ \ \ \ \ By (2.6), we get
   $$\varepsilon \int uv  \leq \frac{|\varepsilon|}{2}\|u\|^2_{L^2} + \frac{|\varepsilon|}{2}\|v\|^2_{L^2}
    \leq \frac{|\varepsilon|}{2\eta}\|u\|^2_{H} + \frac{|\varepsilon|}{2\eta}\|v\|^2_{H} = \frac{|\varepsilon|}{2\eta}\|(u,v)\|^2_{E}.
   $$
 Using (1.5) and the embedding $H^-_b \hookrightarrow L^{r}(\Omega)$ for $r\geq 1$, then for each $\nu >0$, we deduce that
 \begin{eqnarray*}
 \int F(t,x,u) +\int G(t,x,v) &\leq& \nu \|u\|^2_{L^2} + C_\nu \|u\|^{p+1}_{L^{p+1}} + \nu \|v\|^2_{L^2} + C_\nu \|v\|^{q+1}_{L^{q+1}}\\
  &\leq& \frac{\nu}{\eta} \|(u, v)\|^2_{H} + c_1 \|u\|^{p+1}_{H} + c_2 \|v\|^{q+1}_{H}.
 \end{eqnarray*}
 Putting $\nu = \varepsilon/2$, it holds that
  $$\Phi(u,v)\geq \frac{1}{2}\big(1-\frac{2|\varepsilon|}{\eta}\big)\|(u,v)\|_{E}^2 - c_1 \|u\|^{p+1}_{H} - c_2 \|v\|^{q+1}_{H}.
  $$
  \label{}\ \ \ \ Choosing $|\varepsilon| <\eta /2$ and letting
  $$\rho = \min \left\{\Big(\frac{\eta - 2|\varepsilon|}{2c_1 \eta}\Big)^{\frac{1}{p-1}},\ \Big(\frac{\eta - 2|\varepsilon|}{2c_2 \eta}\Big)^{\frac{1}{q-1}}\right\},
  $$
 then for $(u,v)\in E^1$ and $\|(u,v)\|_{E} \leq \rho$, we arrive at
 $$\Phi(u,v)\geq \|u\|^{2}_{H}\Big[\frac{1}{2}\big(1-\frac{2|\varepsilon|}{\eta}\big) - c_1 \|u\|^{p-1}_{H}\Big]
   + \|v\|^{2}_{H}\Big[\frac{1}{2}\big(1-\frac{2|\varepsilon|}{\eta}\big) - c_2 \|v\|^{q-1}_{H}\Big] \geq 0.
 $$
 \label{}\ \ \ \ \ (ii) For $(u,v)\in E^+_b \oplus E^0$, we split $u=u^+ + y$, $v=v^+ + z$, where $u^+$, $v^+ \in H_b^+$, and $y$, $z\in H^0$. Now,
 \begin{eqnarray*}
  \Phi(u,v) &=& -\frac{1}{2}\lag (L+b)u^+, u^+\rag -\frac{1}{2}\lag (L+b)v^+, v^+\rag - \frac{b}{2}\int y^2 - \frac{b}{2}\int z^2 \\
  &-& \varepsilon \int uv - \int F(t,x,u) - \int G(t,x,v).
 \end{eqnarray*}
From (2.6), we get
$$- \int uv \leq \frac{1}{2} \big(\|u^+\|^2_{L^2} + \|y\|^2_{L^2}\big) + \frac{1}{2}\big(\|v^+\|^2_{L^2} + \|z\|^2_{L^2}\big)
\leq \frac{1}{2\eta} \|u^+\|^2_{H} + \frac{1}{2} \|y\|^2_{L^2} + \frac{1}{2\eta} \|v^+\|^2_{H} + \frac{1}{2} \|z\|^2_{L^2}.
$$

 Hence, $F(t,x,u)\geq 0$ and $G(t,x,v)\geq 0$ lead to
  $$\Phi(u,v)\leq -\frac{1}{2}\big(1-\frac{|\varepsilon|}{\eta}\big) \|u^+\|^2_{H} - \frac{1}{2}\big(1-\frac{|\varepsilon|}{\eta}\big) \|v^+\|^2_{H} - \frac{1}{2}\big(b- |\varepsilon|\big) \|y\|^2_{L^2} - \frac{1}{2}\big(b- |\varepsilon|\big) \|z\|^2_{L^2}.
  $$
 Then, for $(u,v)\in E^2$, we have $\Phi(u,v)\leq 0$ by selecting $|\varepsilon| <\min\{\eta, b\}$.
\end{proof}
   {\bf {Proof of \ {\rm (A3)}}:}\\
\label{}\ \ \ \ Concerning with the bound of $\Phi$ in a bounded set, we have:\\[0.3em]
 \textit{{\bf Lemma 3.4.}} \emph{$\Phi$ maps bounded sets into bounded sets.}
\begin{proof}
 Let $R_0>0$ and $D=\{(u,v)\in E~\big|~\|(u,v)\|_E \leq R_0\}$, we claim that there exists a constant $M_0 >0$, such that $\Phi(u,v)\leq M_0$ for each $(u,v)\in D$.\par
   In fact, we decompose $u=u^+ + u^- + y$, $v=v^+ +v^- +z$, where $u^+$, $v^+\in H^+_b$, $u^-$, $v^-\in H^-_b$, and $y$, $z\in H^0$. Then
 combining (2.5), (2.7) with the fact $F(t,x,u)\geq 0$ and $G(t,x,v)\geq 0$, we obtain
 \begin{eqnarray*}
   \Phi(u,v)&\leq& \frac{1}{2}\|u^-\|_{H}^2 + \frac{1}{2}\|v^-\|_{H}^2 - \varepsilon \int uv
  \ \leq \ \frac{1}{2}\|u\|_{H}^2 + \frac{1}{2}\|v\|_{H}^2 + \frac{|\varepsilon|}{2}\|u\|_{L^2}^2 + \frac{|\varepsilon|}{2}\|v\|_{L^2}^2\\
   &\leq& \frac{1}{2}\|(u,v)\|_{E}^2 + \frac{\kappa|\varepsilon|}{2} \|u\|^2_H + \frac{\kappa|\varepsilon|}{2} \|v\|^2_H
   \ \leq\ \frac{1+\kappa|\varepsilon|}{2}R_0^2\ :=\  M_0,
 \end{eqnarray*}
 for each $(u,v)\in D$. That is what we desire.
\end{proof}
   {\bf  {Proof of \ {\rm (A4)}}:}\\
\label{}\ \ \ \ We move to verify that $\Phi$ holds the last condition of Proposition A. \\[0.3em]
 \textit{{\bf Lemma 3.5.}} \emph{For every $m\in {\mathbb {N}}$, and $(u,v)\in E_{m}^1 \oplus E^2$, we have $\Phi(u,v)\rightarrow -\infty$, as $\|(u,v)\|_E \rightarrow \infty$ and $\varepsilon$ is sufficiently small.}
\begin{proof}
  Let $u=u^+ +u^- +y$, $v=v^+ +v^- +z$, for $(u^+, v^+)\in E^+_b$, $(y, z)\in E^0$, and $(u^-, v^-)\in E^1_m = span\{e_1^1, \cdots, e_m^1\}$, where $(e_{n}^1)_{n=1}^{\infty}$ is a basis for $E^1 = E^-_b$.\par
    With the aid of (2.6), the coupled term in (2.5) can be controlled by
    $$-\varepsilon \int uv \leq \frac{|\varepsilon|}{2}\big(\|u\|_{L^2}^2 +\|v\|_{L^2}^2\big) \leq \frac{|\varepsilon|}{2\eta}\big(\|u^+\|_H^2 + \|u^-\|_H^2 + \|v^+\|_H^2 + \|v^-\|_H^2\big)
     + \frac{|\varepsilon|}{2}\big(\|y\|_{L^2}^2 + \|z\|_{L^2}^2\big).
    $$
 \label{}\ \ \ \    To deal with the nonlinear forced terms in (2.5), we utilize (1.3) and the embeddings $L^{p+1}(\Omega)\hookrightarrow L^2(\Omega)$, $L^{q+1}(\Omega)\hookrightarrow L^2(\Omega)$ for $p$, $q>1$, to get
    \begin{eqnarray*}
     -\int F(t,x,u) - \int G(t,x,v) \leq -c_1 \|u\|^{p+1}_{L^{p+1}} - c_1 \|v\|^{q+1}_{L^{q+1}} +c_2
      \leq -c_3 \|u\|^{p+1}_{L^{2}} - c_3 \|v\|^{q+1}_{L^{2}} +c_2.
    \end{eqnarray*}
    Noting the dimension of $E^1_m$ is finite, and the norms in the function space $E_m^1$ are equivalent, then for $(u^- , v^-)$ in $E^1_m$,
    $\|u^-\|_H^{p+1} \leq c \|u^-\|_{L^2}^{p+1} \leq c \|u\|_{L^2}^{p+1}$, and $\|v^-\|_H^{q+1} \leq c \|v\|_{L^2}^{q+1}$. Inserting
 the preceding estimates into (2.5), we have
  \begin{eqnarray*}
   \Phi(u,v) &\leq& -\frac{1}{2}\big(1-\frac{|\varepsilon|}{\eta}\big)\big(\|u^+\|_H^2 + \|v^+\|_H^2 \big)
     - \frac{1}{2}(b-|\varepsilon|)\big(\|y\|_{L^2}^2 + \|z\|_{L^2}^2\big)\\
     &&+\ \frac{1}{2}\big(1+\frac{|\varepsilon|}{\eta}\big)\|u^-\|_H^2 - c_4 \|u^-\|_H^{p+1}
     +\frac{1}{2}\big(1+\frac{|\varepsilon|}{\eta}\big)\|v^-\|_H^2 - c_4 \|v^-\|_H^{q+1} +c_2.
  \end{eqnarray*}
\label{}\ \ \ \  As $\|(u,v)\|_E = \big(\|u\|_H^2 + \|v\|_H^2\big)^{1/2} \rightarrow \infty$, then:
(i) $\|u^+\|_H^2 + \|v^+\|_H^2 + \|y\|_{L^2}^2 + \|z\|_{L^2}^2 \rightarrow \infty$, or (ii) $\|u^-\|_H^2 + \|v^-\|_H^2 \rightarrow \infty$ holds.\par
  If (i) satisfies, then there exists $C>0$, such that
$$\frac{1}{2}\big(1+\frac{|\varepsilon|}{\eta}\big)\|u^-\|_H^2 - c_4 \|u^-\|_H^{p+1}
     +\frac{1}{2}\big(1+\frac{|\varepsilon|}{\eta}\big)\|v^-\|_H^2 - c_4 \|v^-\|_H^{q+1} \leq C,
$$
for $p$, $q>1$. Hence, it follows that $\Phi(u,v)\rightarrow -\infty$ by selecting $|\varepsilon| <\min\{\eta, b\}$, as $\|u^+\|_H^2 + \|v^+\|_H^2 + \|y\|_{L^2}^2 + \|z\|_{L^2}^2 \rightarrow \infty$.\par
   If (ii) holds, then
   $$\frac{1}{2}\big(1+\frac{|\varepsilon|}{\eta}\big)\|u^-\|_H^2 - c_4 \|u^-\|_H^{p+1}
     +\frac{1}{2}\big(1+\frac{|\varepsilon|}{\eta}\big)\|v^-\|_H^2 - c_4 \|v^-\|_H^{q+1} \rightarrow -\infty
$$
 by virtue of $p$, $q>1$. We derive that $\Phi(u,v)\rightarrow -\infty$ as $\|u^-\|_H^2 + \|v^-\|_H^2 \rightarrow \infty$, for  $|\varepsilon| <\min\{\eta, b\}$. The conclusion of Lemma 3.5 is thereby obtained.
\end{proof}
    Now, we have proved that the functional $\Phi \in C^1 (E, {\mathbb{R}})$ satisfy the conditions (A1)$-$(A4) of the Proposition A, which ensure
us to construct a nontrivial critical point $(u, v)$ of $\Phi$ in $E$. Hence, we finish the proof of Theorem 1.1.\hfill $\Box$

\section{Asymptotic behavior of the solutions as $\varepsilon\rightarrow 0$}
\setcounter{equation}{0}
\label{} \ \ \ \ \ In the following, we use $c_i$ and $C$ to denote positive constants which are independent of $\varepsilon$, and whose value may differ from line to line.\par
   Let $(u_\varepsilon, v_\varepsilon)$ be the solution of ${\rm(1.1)}_{a, b, c}$ obtained in Theorem 1.1 for $|\varepsilon|< \varepsilon_0$, we know
 \begin{eqnarray*}
   &&\lag(L+b)u_{\varepsilon}, u_\varepsilon\rag + \varepsilon \int u_\varepsilon v_\varepsilon + \int f(t,x,u_\varepsilon)u_\varepsilon =0,\\
   &&\lag(L+b)v_{\varepsilon}, v_\varepsilon\rag + \varepsilon \int u_\varepsilon v_\varepsilon + \int g(t,x,v_\varepsilon)v_\varepsilon =0.
 \end{eqnarray*}
 Then, by (2.2), (1.3) and (h3) of Theorem 1.1, we have
   \begin{eqnarray}
     \Phi(u_\varepsilon, v_\varepsilon)&=& \frac{1}{2}\int\big[f(t,x,u_\varepsilon)u_\varepsilon-F(t,x,u_\varepsilon)\big]+ \frac{1}{2}\int\big[g(t,x,v_\varepsilon)v_\varepsilon-G(t,x,v_\varepsilon)\big]  \nonumber\\
     &\geq& \frac{p-1}{2}\int F(t,x,u_\varepsilon) + \frac{q-1}{2}\int G(t,x,v_\varepsilon)  \nonumber\\
     &\geq& \frac{(p-1)c_1}{2}\int |u_\varepsilon|^{p+1} + \frac{(q-1)c_1}{2}\int |v_\varepsilon|^{q+1} - 2 c_2 \pi^2.
   \end{eqnarray}
   \indent On the other hand, we deduce from Remark 2.1 together with Lemma 3.4 and Lemma 3.5 that, there exists a positive number $c_3$ independent of $\varepsilon$, such that
   \begin{eqnarray}
     \Phi(u_\varepsilon, v_\varepsilon) \leq c_3,\ \ \ \ \ \ \ {\rm for\ any}\ \varepsilon \in (-\varepsilon_0, \varepsilon_0).
   \end{eqnarray}
 \indent Combining with (4.1), (4.2) and by virtue of $p$, $q>1$, we get
 \begin{eqnarray}
     \|u_\varepsilon\|_{L^{p+1}} + \|v_\varepsilon\|_{L^{q+1}} \leq C,\ \ \ \ \ {\rm for\ any}\ \varepsilon \in (-\varepsilon_0, \varepsilon_0).
   \end{eqnarray}

 Let $\varepsilon_n \in (-\varepsilon_0, \varepsilon_0)$ be any sequence with $\varepsilon_n \rightarrow 0$ as $n\rightarrow \infty$. Subsequently, we will prove that $(u_{\varepsilon_n}, v_{\varepsilon_n})$ converge strongly to some $(U_0, V_0)$ in $L^2 (\Omega) \times L^2 (\Omega)$ as $n\rightarrow \infty$, by passing to a subsequence, and justify that $U_0$, $V_0$ are weak solutions of the scalar wave equations (W1) and (W2) respectively.\par
    At first, we decompose $u_{\varepsilon_n}^{\ } = u^+_{\varepsilon_n} + u^-_{\varepsilon_n} + y^{\ }_{\varepsilon_n}$, $v_{\varepsilon_n}^{\ } = v^+_{\varepsilon_n} + v^-_{\varepsilon_n} + z^{\ }_{\varepsilon_n}$, where $(u^+_{\varepsilon_n}, v^+_{\varepsilon_n}) \in E^+_b$, $(u^-_{\varepsilon_n}, v^-_{\varepsilon_n}) \in E^-_b$, and $(y^{\ }_{\varepsilon_n}, z^{\ }_{\varepsilon_n}) \in E^0$, and show the next lemma concerning with the asymptotic behavior of $(u_{\varepsilon_n}, v_{\varepsilon_n})$.\\[0.3em]
 \textit{{\bf Lemma 4.1.}} \emph{Passing to a subsequence of $\varepsilon_n \rightarrow 0$ as $n\rightarrow \infty$, we have}\par
 (i) $(u^+_{\varepsilon_n}, v^+_{\varepsilon_n})$ \emph{converge strongly to some} $(U^+_0,V^+_0)$ \emph{in} $E^+_b$;\par
 (ii) $(u^-_{\varepsilon_n}, v^-_{\varepsilon_n})$ \emph{converge strongly to some} $(U^-_0,V^-_0)$ \emph{in} $E^-_b$;\par
 (iii) $(y^{\ }_{\varepsilon_n}, z^{\ }_{\varepsilon_n})$ \emph{converge strongly to some} $(U^0_0,V^0_0)$ \emph{in} $E^0$;\par
  (iv) $u^{\ }_{\varepsilon_n}\rightharpoonup U_0$  \emph{weakly in} $L^{p+1}(\Omega)$, \emph{and} $v^{\ }_{\varepsilon_n}\rightharpoonup V_0$  \emph{weakly in $L^{q+1}(\Omega)$, where $U_0 = U_0^+ +U^-_0 + U^0_0$, $V_0 = V_0^+ +V^-_0 + V^0_0$.}
\begin{proof} First, we establish the uniform bound for $\{(u_{\varepsilon_n}, v_{\varepsilon_n})\}$ in $E$.\par
   Recording (4.3) and the embedding properties of $L^{r}(\Omega)\hookrightarrow L^2(\Omega)$ for $r\geq 2$, we have
  \begin{eqnarray}
    \|y^{\ }_{\varepsilon_n}\|^2_{L^{2}} + \|z^{\ }_{\varepsilon_n}\|^2_{L^{2}} \leq \|u_{\varepsilon_n}\|^2_{L^{2}} + \|v_{\varepsilon_n}\|^2_{L^{2}} \leq c_4 \|u_{\varepsilon_n}\|^2_{L^{p+1}} + c_4 \|v_{\varepsilon_n}\|^2_{L^{q+1}} \leq C.
  \end{eqnarray}

  For $(u^+_{\varepsilon_n}, v^+_{\varepsilon_n}) \in E^+_b$, by (h1), (2.8), (4.3) and the orthogonality of $H_b^+$, $H_b^-$, $H^0$,
   \begin{eqnarray}
  && \|u^+_{\varepsilon_n}\|_H ^2 + \|v^+_{\varepsilon_n}\|_H ^2\ =\ \lag(L+b)u_{\varepsilon_n}, u^+_{\varepsilon_n}\rag + \lag(L+b)v_{\varepsilon_n}, v^+_{\varepsilon_n}\rag \nonumber\\
  &=&  - 2\varepsilon_n \int u^+_{\varepsilon_n} v^+_{\varepsilon_n} - \int f(t,x,u_{\varepsilon_n})u^+_{\varepsilon_n} - \int g(t,x,v_{\varepsilon_n})v^+_{\varepsilon_n}\nonumber \\
  &\leq& c_5\big( \|u_{\varepsilon_n}\|_{L^{2}} \|v_{\varepsilon_n}\|_{L^{2}} +  \|u_{\varepsilon_n}\|_{L^2} +  \|u_{\varepsilon_n}\|_{L^{p+1}}^{p} \|u^+_{\varepsilon_n}\|_{H} +  \|v_{\varepsilon_n}\|_{L^2} + \|v_{\varepsilon_n}\|_{L^{q+1}}^{q} \|v^+_{\varepsilon_n}\|_{H}\big)\nonumber  \\
  &\leq& C(1+\|(u^+_{\varepsilon_n}, v^+_{\varepsilon_n})\|_E).
 \end{eqnarray}

Similarly, we get
 \begin{eqnarray}
  && \|u^-_{\varepsilon_n}\|_H ^2 + \|v^-_{\varepsilon_n}\|_H ^2\ =\ -\lag(L+b)u_{\varepsilon_n}, u^-_{\varepsilon_n}\rag - \lag(L+b)v_{\varepsilon_n}, v^-_{\varepsilon_n}\rag \nonumber\\
  &\leq& C(1+\|(u^-_{\varepsilon_n}, v^-_{\varepsilon_n})\|_E),\qquad\qquad {\rm for}\ (u^-_{\varepsilon_n}, v^-_{\varepsilon_n}) \in E^-_b.
 \end{eqnarray}

  Therefore, summing up $(4.4)$-$(4,6)$, we arrive at
  $$\|(u_{\varepsilon_n}, v_{\varepsilon_n})\|^2_E \leq C(1+\|(u_{\varepsilon_n}, v_{\varepsilon_n})\|_E),
  $$
which implies that
\begin{eqnarray}
\|(u_{\varepsilon_n}, v_{\varepsilon_n})\|_E \leq C_0, \ \ \ \ {\rm where}\ C_0>0 \ {\rm is\  independent\ of}\ \varepsilon_n.
 \end{eqnarray}
Then, $\{(u_{\varepsilon_n}^{\ }, v_{\varepsilon_n}^{\ })\}$ converge weakly to some $(U_0,V_0)\in E$, by passing to a subsequence of $\varepsilon_n \rightarrow 0$ as $n\rightarrow \infty$. Spilt $U_0 = U^+_0 + U^-_0 + U^0_0$, $V_0 = V^+_0 + V^-_0 + V^0_0$, where $(U^+_0,V^+_0) \in E^+_b$, $(U^-_0,V^-_0) \in E^-_b$ and $(U^0_0,V^0_0) \in E^0$. We assert that $(U^+_0,V^+_0)$, $(U^-_0,V^-_0)$ and $(U^0_0,V^0_0)$ satisfy (i), (ii), (iii) of this lemma.\par
   Since $\Phi'(u_{\varepsilon_n}, v_{\varepsilon_n})=0$, and by the orthogonality of $H_b^+$, $H_b^-$, $H^0$, it follows that
\begin{eqnarray*}
 && \|u^+_{\varepsilon_n} - U^+_0\|_H ^2 = \lag(L+b)(u^+_{\varepsilon_n}-U^+_0), u^+_{\varepsilon_n}-U^+_0\rag\nonumber \\
   &=& \lag(L+b)u_{\varepsilon_n}, u^+_{\varepsilon_n}-U^+_0\rag - \lag(L+b)U_0, u^+_{\varepsilon_n}-U^+_0\rag \nonumber \\
   &=& -\varepsilon_n \int v_{\varepsilon_n} (u^+_{\varepsilon_n}-U^+_0) - \int f(t,x,u_{\varepsilon_n})(u^+_{\varepsilon_n}-U^+_0)  - \lag(L+b)U_0, u^+_{\varepsilon_n}-U^+_0\rag.
\end{eqnarray*}
   By the weak convergence of $u^+_{\varepsilon_n} \rightharpoonup U^+_0$ as $n\rightarrow \infty$, we can proceed as in the proof of (3.10), (3.13) and (3.15) in Lemma 3.2 to show that $u^+_{\varepsilon_n} \rightarrow U^+_0$ strongly in $E^+_b$, as $n\rightarrow \infty$. Then (i) holds. We can prove (ii) in the same way.\par
   For $y^{\ }_{\varepsilon_n} \in H^0$, $z^{\ }_{\varepsilon_n} \in H^0$,
 \begin{eqnarray*}
  && -\lag \Phi'(u^{\ }_{\varepsilon_n}, v^{\ }_{\varepsilon_n}),\ (y^{\ }_{\varepsilon_n} - U_0^0, z^{\ }_{\varepsilon_n} - V_0^0)\rag\\
   &=& \lag Lu^{\ }_{\varepsilon_n},  y^{\ }_{\varepsilon_n}-U_0^0 \rag
     + b\int u^{\ }_{\varepsilon_n} (y^{\ }_{\varepsilon_n}-U_0^0) +\varepsilon_n \int v^{\ }_{\varepsilon_n} (y^{\ }_{\varepsilon_n}-U_0^0)+ \int f(t,x,u^{\ }_{\varepsilon_n})(y^{\ }_{\varepsilon_n} - U_0^0)\\
  && +\ \lag Lv^{\ }_{\varepsilon_n},  z^{\ }_{\varepsilon_n}-V_0^0 \rag   + b\int v^{\ }_{\varepsilon_n} (z^{\ }_{\varepsilon_n}-V_0^0) +\varepsilon_n \int u^{\ }_{\varepsilon_n} (z^{\ }_{\varepsilon_n}-V_0^0)+ \int g(t,x,v^{\ }_{\varepsilon_n})(z^{\ }_{\varepsilon_n} - V_0^0).
  \end{eqnarray*}
By virtue of $\Phi'(u^{\ }_{\varepsilon_n}, v^{\ }_{\varepsilon_n})=0$ and $\lag Lu^{\ }_{\varepsilon_n},  y^{\ }_{\varepsilon_n}-U_0^0 \rag = \lag Lv^{\ }_{\varepsilon_n},  z^{\ }_{\varepsilon_n}-V_0^0 \rag =0$, we have
 \begin{eqnarray*}
  && b\int u^{\ }_{\varepsilon_n} (y^{\ }_{\varepsilon_n}-U_0^0) +\varepsilon_n \int v^{\ }_{\varepsilon_n} (y^{\ }_{\varepsilon_n}-U_0^0) + b\int v^{\ }_{\varepsilon_n} (z^{\ }_{\varepsilon_n}-V_0^0) +\varepsilon_n \int u^{\ }_{\varepsilon_n} (z^{\ }_{\varepsilon_n}-V_0^0) \\
   &=& - \int f(t,x,u^{\ }_{\varepsilon_n})(y^{\ }_{\varepsilon_n} - U_0^0) - \int g(t,x,v^{\ }_{\varepsilon_n})(z^{\ }_{\varepsilon_n} - V_0^0).
  \end{eqnarray*}
 By an analogue proof of $(3.18)-(3.21)$ in Lemma 3.2, we deduce that
 $$\|y^{\ }_{\varepsilon_n}-U^0_0\|^2_{L^2} + \|z^{\ }_{\varepsilon_n}-V^0_0\|^2_{L^2} \rightarrow 0,\ \ \ \ {\rm{as}}\ n\rightarrow \infty,
 $$
and (iii) is satisfied.\par
    As a consequence of (i), (ii), (iii), we have $\{(u_{\varepsilon_n}^{\ }, v_{\varepsilon_n}^{\ })\}$ converge strongly to
$(U_0,V_0)$ in $E$. Particularly,
\begin{eqnarray}
(u_{\varepsilon_n}^{\ },\ v_{\varepsilon_n}^{\ })\rightarrow (U_0,\ V_0)\ {\rm strongly\ in}\ L^2(\Omega)\times L^2(\Omega),\ \ {\rm as} \ n\rightarrow \infty.
\end{eqnarray}

   On the other hand, (4.3) implies that there exists some $U_1 \in L^{p+1}(\Omega)$, such that $\{u_{\varepsilon_n}^{\ }\}$ possesses a subsequence which converge weakly to $U_1$ in $L^{p+1}(\Omega)$. By virtue of $L^{p+1}(\Omega)\hookrightarrow L^{2}(\Omega)$ for $p>1$, and noting the fact that the weak limit of $\{u_{\varepsilon_n}^{\ }\}$ is unique, we have $U_1=U_0$.
Thus, $u_{\varepsilon_n}^{\ } \rightharpoonup U_0$ weakly in $L^{p+1}(\Omega)$. Similarly, we get $v_{\varepsilon_n}^{\ } \rightharpoonup V_0$ weakly in $L^{q+1}(\Omega)$. Hence, (iv) holds.
\end{proof}

  Consequently, we assert \\[0.3em]
 \textit{{\bf Lemma 4.2.}} \emph{For any $\varphi \in H\cap L^{p+1}(\Omega)$, $\psi \in H\cap L^{q+1}(\Omega)$, then passing to a subsequence of $\varepsilon_n \rightarrow 0$ as $n\rightarrow \infty$, we have}
   \begin{eqnarray}
  && \lag(L+b)u_{\varepsilon_n},\ u_{\varepsilon_n} - \varphi \rag \rightarrow \lag(L+b)U_0,\ U_0 - \varphi \rag, \\
  && \lag(L+b)v_{\varepsilon_n},\ v_{\varepsilon_n} - \psi \rag \rightarrow \lag(L+b)V_0,\ V_0 - \psi \rag; \\
 && \lag f(t,x,\varphi),\ u_{\varepsilon_n} - \varphi \rag \rightarrow \lag f(t,x,\varphi),\ U_0 - \varphi \rag, \\
 && \lag g(t,x,\psi),\ v_{\varepsilon_n} - \psi \rag \rightarrow \lag g(t,x,\psi),\ V_0 - \psi \rag; \\
 && \lag\varepsilon_n  v_{\varepsilon_n},\ u_{\varepsilon_n} - \varphi \rag \rightarrow 0,\ \ {\rm and}\ \ \lag \varepsilon_n  u_{\varepsilon_n},\ v_{\varepsilon_n} - \psi \rag \rightarrow 0.
   \end{eqnarray}
 \begin{proof}
   To reach (4.9), we write
   \begin{eqnarray}
     && \lag(L+b)u_{\varepsilon_n},\ u_{\varepsilon_n} - \varphi \rag - \lag(L+b)U_0,\ U_0 - \varphi \rag \nonumber\\
   &=& \lag(L+b)u_{\varepsilon_n},\ u_{\varepsilon_n} - U_0 \rag + \lag(L+b)(u_{\varepsilon_n}-U_0),\ U_0 - \varphi \rag := A_1 + A_2.
   \end{eqnarray}
   From  the orthogonality of $H_b^+$, $H_b^-$, $H^0$, it follows that
   \begin{eqnarray*}
     A_1 = \lag(L+b)u^+_{\varepsilon_n},\ u^+_{\varepsilon_n} - U^+_0 \rag + \lag(L+b)u^-_{\varepsilon_n},\ u^-_{\varepsilon_n} - U^-_0 \rag + b \lag y_{\varepsilon_n},\ y_{\varepsilon_n} - U^0_0 \rag.
   \end{eqnarray*}
   Then, by (4.7), Lemma 4.1 and using H\"{o}lder inequality, we infer
   \begin{eqnarray*}
     |A_1| \leq \|u^+_{\varepsilon_n}\|_{H} \|u^+_{\varepsilon_n}-U^+_0\|_{H} + \|u^-_{\varepsilon_n}\|_{H} \|u^-_{\varepsilon_n}-U^-_0\|_{H} + b \|y_{\varepsilon_n}\|_{L^2} \|y_{\varepsilon_n}-U^0_0\|_{L^2} \rightarrow 0,
   \end{eqnarray*}
   as $n\rightarrow \infty$. Furthermore, since $\varphi \in H$, we deduce that $A_2 \rightarrow 0$ in the same way. Thus, (4.9) holds, and we can obtain (4.10) similarly.\par
     We come to prove (4.11). By condition (h1) of Theorem 1.1, we have
     \begin{eqnarray}
       |\lag f(t,x,\varphi),\ u_{\varepsilon_n} - U_0 \rag| \leq c_0 \int |u_{\varepsilon_n} - U_0| +c_0 \int |\varphi|^p |u_{\varepsilon_n} - U_0|.
     \end{eqnarray}
  Then (4.8), (iv) of Lemma 4.1 and $|\varphi|^p \in L^{\frac{p+1}{p}}(\Omega) = (L^{p+1}(\Omega))^*$ show the right hand side of (4.15) go to zero as $n\rightarrow \infty$, which gives (4.11). Furthermore, (4.12) also holds for any $\psi \in H\cap L^{q+1}(\Omega)$.\par
   At last, by (4.7) and H\"{o}lder inequality, it follows that
    \begin{eqnarray*}
     &&|\lag\varepsilon_n  v_{\varepsilon_n},\ u_{\varepsilon_n} - \varphi \rag| + |\lag \varepsilon_n  u_{\varepsilon_n},\ v_{\varepsilon_n} - \psi \rag|\\
     &\leq& |\varepsilon_n| \|v_{\varepsilon_n}\|_{L^2} \|u_{\varepsilon_n} - \varphi\|_{L^2} + |\varepsilon_n| \|u_{\varepsilon_n}\|_{L^2} \|v_{\varepsilon_n} - \psi\|_{L^2}\ \leq C|\varepsilon_n|\rightarrow 0,
    \end{eqnarray*}
  as $n\rightarrow \infty$. Hence, we arrive at (4.13).
 \end{proof}

   We are ready for the {\bf proof of Theorem 1.3:}
\begin{proof}
 We are suffice to prove that
\begin{eqnarray}
 \lag (L+b)U_0 + f(t,x,U_0),\ \omega\rag = 0,&&\quad \forall~ \omega \in H\cap L^{p+1}(\Omega),\\
 \lag (L+b)V_0 + g(t,x,V_0),\ \chi\rag = 0,&&\quad \forall~ \chi \in H\cap L^{q+1}(\Omega).
\end{eqnarray}

  As $(u_{\varepsilon_n}, v_{\varepsilon_n})$ solves the problem (1.1)$_{a,b,c}$ with linear coupling constant $\varepsilon = \varepsilon_n$, then
  \begin{eqnarray}
   \lag (L+b)u_{\varepsilon_n} + f(t,x,u_{\varepsilon_n}) + \varepsilon_n v_{\varepsilon_n}, \ u_{\varepsilon_n}-\varphi \rag =0,\ \ \ \ \forall~ \varphi \in H\cap L^{p+1}(\Omega).
  \end{eqnarray}
By (h4) of Theorem 1.1, we know
$$\lag f(t,x,u_{\varepsilon_n})- f(t,x,\varphi),\ u_{\varepsilon_n}-\varphi \rag \geq 0.
$$
Hence, we derive from (4.18) that
\begin{eqnarray}
   \lag (L+b)u_{\varepsilon_n} + f(t,x,\varphi) + \varepsilon_n v_{\varepsilon_n}, \ u_{\varepsilon_n}-\varphi \rag \leq 0,\ \ \ \forall~ \varphi \in H\cap L^{p+1}(\Omega).
  \end{eqnarray}

  By (4.9), (4.11) and (4.13) of Lemma 4.2, then passing to the limit in (4.19), we have
 \begin{eqnarray}
 \lag (L+b)U_0 + f(t,x,\varphi),\ U_0 -\varphi\rag \leq 0, \ \ \ \forall~ \varphi \in H\cap L^{p+1}(\Omega).
\end{eqnarray}
Choosing $\varphi= U_0 - \lambda \omega$ with $\lambda>0$ in (4.20), then dividing by $\lambda$ and letting $\lambda\rightarrow 0$, we obtain
\begin{eqnarray*}
 \lag (L+b)U_0 + f(t,x,U_0),\ \omega\rag \leq 0,
\end{eqnarray*}
and noting $\omega \in H\cap L^{p+1}(\Omega)$ is chosen arbitrarily, we infer that $U_0$ satisfies
\begin{eqnarray*}
 \lag (L+b)U_0 + f(t,x,U_0),\ \omega\rag = 0, \ \ \ \forall~ \omega \in H\cap L^{p+1}(\Omega).
\end{eqnarray*}
Thereby, (4.16) is concluded.\par
  In the same manner, we obtain (4.17) by virtue of (4.10), (4.12) and (4.13). This concludes the proof of Theorem 1.3.
\end{proof}

\section{Higher regularity of the solutions}
\setcounter{equation}{0}
\label{}\ \ \ \ In this section, we prove the solutions $(u,v)\in E$ obtained in the previous section enjoy the higher regularity, providing that $\varepsilon$ is sufficiently small.

\subsection{A representation theorem for single wave equation}
\label{}\ \ \ \ \
   To proceed, we collect some facts that are useful in improving the regularity of the weak solutions for (1.1)$_{a,b,c}$.\par
   Let $\Omega=[0,2\pi]\times [0,\pi]$, and $\ker L$ is the kernel of the d'Alembert operator $L=\partial_{t} ^2 - \partial_{x} ^2$.
Denote $R(L)$ is the range of $L$, then we have $R(L)=(\ker L)^\bot$. The following representation theorem plays an important role in the study of the regularity theory of scalar wave equations.
  \begin{prop}
  {\rm(See \cite{Bre-Cor-Nir})} Given $h\in L^1 (\Omega) \cap R(L)$ satisfying $h(t+2\pi, x)=h(t,x)$, then the solution of the wave equation
   \begin{equation*}
  \qquad\qquad\quad\qquad\qquad\left\{\begin{array}{l}
    Lw\equiv w_{tt}-w_{xx}=h(t,x),\;\  (t,x)\in \Omega,
    \\[1.2ex]
     w(t,0)=w(t,\pi)=0, \;\ w(t+2\pi,x)=w(t,x),
      \end{array}\right.   \qquad\qquad\qquad {\rm(W)}
  \end{equation*}
 can be presented in the form of $w=w_0 + w_1$, where $w_0 \in \ker L$, and $w_1 \in (\ker L)^\bot = R(L)$. More precisely, we have
  $$w_0 (t,x)= p(t+x)-p(t-x),
  $$
   for some $p\in L^1$ which is $2\pi$-periodic and satisfying $\int_0^{2\pi} p(\tau) \dif \tau =0$, and
 $$w_1 (t,x)=-\frac{1}{2} \int_{x}^{\pi}\dif \xi \int_{t+x-\xi}^{t-x+\xi} h(\tau, \xi)\dif \tau +
  \frac{\pi - x}{2\pi} \int_{0}^{\pi}\dif \xi \int_{t-\xi}^{t+\xi} h(\tau, \xi)\dif \tau.
 $$
 \end{prop}
 \textit{Estimates of the component $w_1$ in $R(L)=(\ker L)^\bot$}: \par\medskip
  From Proposition 5.1, the $L^\infty$-norm of $w_1 \in R(L)$ can be controlled by
  \begin{eqnarray}
  \|w_1\|_{L^\infty}\leq c\|h\|_{L^1}, \ \ {\rm for}\ h\in L^1\cap R(L).
  \end{eqnarray}
Furthermore, if $h\in L^\infty (\Omega)\cap R(L)$, then we have
 \begin{eqnarray}
 \|w_1\|_{C^{0,1}}\leq c\|h\|_{L^\infty},
  \end{eqnarray}
where $\|\cdot\|_{C^{0,1}}$ is the norm of the Lipschitz space $C^{0,1}(\Omega)$.\par\medskip
 \textit{An Integral Formula concerning with the component in $\ker L$}: \par\medskip
   If $p(s)$, $q(s)$ are $L^1$ functions with period $2\pi$ and satisfy $\int_0^{2\pi} p(s) \dif s =\int_0^{2\pi} q(s) \dif s =0$, then a computation as in \cite{Ra} shows that
   \begin{eqnarray}
   \iint_{\Omega}p(t+x)q(t-x)\dif t \dif x =0.
   \end{eqnarray}
\label{}\ \ \ \    We also require the next property to characterize \textit{the range of the operator} $L$:\par
    A function $h$ belongs to $R(L)$ if and only if
    \begin{eqnarray}
    \int_{0}^\pi [h(t+x,x)-h(t-x,x)]\dif x =0.
    \end{eqnarray}
In other words, the sufficient and necessary condition for the solvability of linear wave equation (W) is that the function $h(t,x)$ satisfies (5.4).  We see \cite{Bre-Cor-Nir, Lo} for more details.\par
\subsection{Proof of Theorem 1.4}
 \label{}\ \ \ \  Let $(u,v)\in E$ be a solution of (1.1)$_{a,b,c}$ constructed by local linking method. We decompose it into
$(u,v)=(u_1 +y, v_1 +z)$, where $u_1$, $v_1\in H_{b}^+ \oplus H_{b}^- \equiv R(L)$,  and $y$, $z\in H^0 \equiv \ker L$. We apply Proposition 5.1 to represent $y,\ z$ by
$$y= p(t+x)- p(t-x)\ \ {\rm and}\ \ z= q(t+x)- q(t-x),
$$
 for some $2\pi$-periodic
functions $p,\ q\in L^1 ([0, 2\pi])$ such that $\int_0^{2\pi} p(\tau) \dif \tau =\int_0^{2\pi} q(\tau) \dif \tau =0$.\par\medskip
\textit{The $L^\infty-$regularity of the components in $H_{b}^+ \oplus H_{b}^-$}:\par\smallskip
    Since $(u,v)$ satisfy the system
      \begin{equation*}
  \left\{\begin{array}{l}
     u_{tt}- u_{xx}  + bu + \varepsilon v + f(t,x,u) = 0,
    \\[1.2ex]
     v_{tt}- v_{xx} + bv + \varepsilon u + g(t,x,v) = 0,
 \end{array} \right. \;\quad  (t,x)\in {\Omega},
  \end{equation*}
we shall derive the $L^\infty -$ estimate of $u_1$, $v_1 \in H_{b}^+ \oplus H_{b}^-$ from Proposition 5.1. We have:
\begin{lem}
  Assume that $(u,v)$ is a solution of {\rm (1.1)}$_{a,b,c}$ obtained in {\rm Theorem 1.1}, then there exists $d_1$, $d_2 >0$, such that $\|u_1\|_{L^\infty} \leq d_1$ and $\|v_1\|_{L^\infty} \leq d_2$.
\end{lem}
\begin{proof}
 By (2.8), $p>1$ and (h1) of Theorem 1.1, we have
   $$\int |f(t,x,u)| \leq c_0 \int (1+|u|^p) \leq c_1 (1+\|u\|_E ^p) <\infty,
   $$
which means $f\in L^1 (\Omega\times \mathbb{R})$. Noting that (2.8) also implies $u$, $v\in L^1(\Omega)$, then we can use the estimate (5.1) to obtain a number $d_1 >0$, such that
\begin{eqnarray}
\|u_1\|_{L^\infty} \leq c\|bu + \varepsilon v + f(t,x,u)\|_{L^1}\leq d_1< \infty.
\end{eqnarray}
Similarly, we have
  \begin{eqnarray}
  \|v_1\|_{L^\infty}\leq c\|bv + \varepsilon u + g(t,x,v)\|_{L^1}\leq d_2< \infty,
  \end{eqnarray}
  for some $d_2 >0$.
\end{proof}

\textit{The $L^\infty-$regularity of the components in $H^0 \equiv \ker L$}:\par\smallskip
    For the terms of $y$, $z$ in $H^0$, the proof of $y$, $z\in L^\infty (\Omega)$ is a difficult task since the a-prior estimate
(5.1) is invalid for $y$ and $z$. To this end, we set
    \begin{eqnarray}
    N_1= \|p\|_{L^\infty (\Omega)},\ \  N_2= \|q\|_{L^\infty (\Omega)}.
    \end{eqnarray}
    Without loss of generality, we may suppose that there are $s_1$ and $s_2$, such that $p(s_1)>N_1 -1$ and $q(s_2)>N_2 -1$.  We shall derive the upper bounds of $N_1$
and $N_2$. \par
    Let $\tilde{f}(t,x) = f(t,x,u(t,x))$ and
   \begin{eqnarray}
   h_1 (t,x)= bu(t,x)+\varepsilon v (t,x) +\ \tilde{f}\left(t,x\right),
   \end{eqnarray}
   where $u(t,x)= u_1(t,x)+ p(t+x)- p(t-x)$, $v(t,x)= v_1(t,x)+ q(t+x)- q(t-x)$.\par
   We prepare the following two lemmas to prove Theorem 1.4.
   \begin{lem}
    Assume that $(u(t,x),v(t,x))$ is a solution of {\rm (1.1)}$_{a,b,c}$ obtained in {\rm Theorem 1.1}, then there exists a number $d_3 >0$, such that
     \begin{eqnarray}
       \int_{0}^\pi \tilde{f} (s_1 -x, x) \dif x \leq d_3 - b\pi p(s_1) -2\varepsilon \pi q(s_1).
     \end{eqnarray}
   \end{lem}

 \begin{proof}
  Noting that $Lu + h_1(t,x)=0$, then putting $t=s_1$ in (5.4), we have
  \begin{eqnarray}
  \int_{0}^\pi h_1(s_1 +x,x)\dif x =  \int_{0}^\pi h_1(s_1 -x,x)\dif x .
    \end{eqnarray}

   Since $p(s_1 +2x)-p(s_1)\leq 1$ and $u_1(s_1+x,x)\leq d_1$, it follows from (h4), (5.5) and  $f\in C (\Omega \times \mathbb{R},~\mathbb{R})$ that
   $$\tilde{f}\left(s_1+x,x\right)\leq f\left(s_1+x,x,d_1 +1\right) \leq d_4,
   $$
   for some $d_4>0$. Therefore, by virtue of $\int_0 ^\pi q(s_1 +2x) \dif x =0$, we deduce from (5.5), (5.6) and the above inequality that
  \begin{eqnarray}
    \int_{0}^\pi h_1(s_1 +x,x)\dif x  &\leq&  \int_{0}^\pi \big[bd_1 + |\varepsilon| d_2 +\varepsilon q(s_1+2x) -\varepsilon q(s_1) +d_4\big] \dif x \nonumber\\[0.4em]
    &\leq& d_5 - \varepsilon \pi q(s_1).
  \end{eqnarray}
\indent On the other hand, it follows from (5.5), (5.6) that
$$u_1(s_1 -x,x)\geq -d_1\ \  {\rm and}\ \ v_1(s_1 -x,x)\geq -d_2.
 $$
 Furthermore, by the fact of $\int_0 ^\pi p(s_1 -2x) \dif x= \int_0 ^\pi q(s_1 -2x) \dif x =0$, we have
   \begin{eqnarray}
     \int_{0}^\pi h_1(s_1 -x,x)\dif x
    \geq b\pi p(s_1)+\varepsilon \pi q(s_1)+\int_0 ^\pi \tilde{f}\left(s_1-x,x\right) \dif x -b\pi d_1.
   \end{eqnarray}
\label{}\ \ \ \ Hence, combining with (5.10)-(5.12), we arrive at (5.9).
\end{proof}

To proceed further, we denote by
 $$\Lambda_1 = \left\{x\in [0,\pi]: p(s_1 -2x)\leq \frac{N_1}{2}\right\}\ \ {\rm and}\ \ \Lambda_2 = \left\{x\in [0,\pi]: q(s_2 -2x)\leq \frac{N_2}{2}\right\}.
 $$
 We have
 \begin{lem}
     Assume that $(u(t,x),v(t,x))$ is a solution of {\rm (1.1)}$_{a,b,c}$ obtained in {\rm Theorem 1.1}, then {\rm meas}$\Lambda_1 \geq \pi/3$, {\rm meas}$\Lambda_2 \geq \pi/3$ and there exists a number $d>0$, such that
      \begin{eqnarray}
    \int_{\Lambda_1} f\left(s_1-x,x, \frac{N_1}{2}-d_1 -1\right) \dif x + \int_{\Lambda_2} g\left(s_2-x,x, \frac{N_2}{2}-d_2 -1\right) \dif x \leq d.
  \end{eqnarray}
 \end{lem}
 \begin{proof}
 (1) At first, we prove that {\rm meas}$\Lambda_1 \geq \pi/3$.
 In view of
     \begin{equation*}
  p(s_1-2x)\geq\left\{\begin{array}{l}
     -N_1,\ \ {\rm for}\ x\in \Lambda_1,
    \\[1.2ex]
     N_1/2,\ \ {\rm for}\ x\in \Lambda_1^c,
 \end{array} \right.
  \end{equation*}
  we get
    \begin{eqnarray*}
      0 &=& \int_0 ^\pi p(s_1 -2x)\dif x = \int_{\Lambda_1}p(s_1 -2x)\dif x + \int_{\Lambda_1 ^c}p(s_1 -2x)\dif x \\
       &\geq&  -N_1 {\rm meas}\Lambda_1 + \frac{N_1}{2}\left(\pi- {\rm meas}\Lambda_1\right),
    \end{eqnarray*}
 which indicates ${\rm meas}\Lambda_1 \geq \pi/3$. \par
 Similarly, we also have ${\rm meas}\Lambda_2 \geq \pi/3$.\par
(2) Now we turn to the proof of (5.13). To this end, we represent the left hand side of (5.9) in the form of
 $$\left(\int_{\Lambda_1}+\int_{\Lambda_1 ^c}\right) \tilde{f}\left(s_1-x,x\right) \dif x := \widetilde{Q}_1 + \widetilde{Q}_2 .
 $$
\label{}\ \ \ \ It follows from (5.7) that $p(s_1)-p(s_1 -2x)\geq N_1/2 -1$, for $x\in \Lambda_1$. Then by (h4) and (5.5), we know
  $$\widetilde{Q}_1 \geq \int_{\Lambda_1} f\left(s_1-x,x, \frac{N_1}{2}-d_1 -1\right) \dif x.
  $$
  On the other hand, recording $u_1(s_1-x,x)+p(s_1)-p(s_1 -2x) \geq -d_1 -1$, then we deduce from (h1) and (h4) that
   $$\tilde{f}\left(s_1-x,x\right) \geq f\left(s_1-x,x,-d_1-1\right)\geq -c_0-c_0 d_1^p,
   $$
   which gives
  $$\widetilde{Q}_2 \geq -\int_{\Lambda_1 ^c} \left(c_0 +c_0 d_1 ^p\right) \dif x \geq -d_6.
  $$
  Thus, coming back to (5.9), we get
  \begin{eqnarray}
    \int_{\Lambda_1} f\left(s_1-x,x, \frac{N_1}{2}-d_1 -1\right) \dif x
     &\leq& \int_0^\pi \tilde{f}\left(s_1-x,x\right) \dif x - \widetilde{Q}_2 \nonumber\\[0.5em]
    &\leq& d_7 - b\pi p(s_1)- 2\varepsilon \pi q(s_1).\quad
  \end{eqnarray}
 \label{}\ \ \ \ Applying a similar procedure, we also have
    \begin{eqnarray}
    \int_{\Lambda_2} g\left(s_2-x,x, \frac{N_2}{2}-d_2 -1\right) \dif x \leq d_7 ' - b\pi q(s_2)- 2\varepsilon \pi p(s_2).
  \end{eqnarray}

 Choosing $|\varepsilon| <b/2$, we infer from (5.7) that $|p(s_2)|\leq p(s_1) +1$, $|q(s_1)|\leq q(s_2) +1$, which assure $ b\pi p(s_1) + 2\varepsilon \pi p(s_2) \geq -2\varepsilon \pi$ and $ b\pi q(s_2) + 2\varepsilon \pi q(s_1) \geq -2\varepsilon \pi$. Then, adding up (5.14) and (5.15), we obtain (5.13).
  \end{proof}

 Now we are ready to prove Theorem 1.4.\\[0.4em]
 {\bf Proof of Theorem 1.4.}
 Without loss of generality, we may assume that $N_1 >2d_1 +2$ and $N_2 >2d_2 +2$. Thus, we follow from (h2) and (h4) to derive
 $$\int_{\Lambda_1} f\left(s_1-x,x, \frac{N_1}{2}-d_1 -1\right) \dif x \geq 0,\ \ \int_{\Lambda_2} g\left(s_2-x,x, \frac{N_2}{2}-d_2 -1\right) \dif x \geq 0.
 $$

    Now, we claim that $N_1 <+\infty$. If not, from the fact of $f(s_1-x,x,\xi)\rightarrow +\infty$ as $\xi\rightarrow +\infty$, for $x\in \Lambda_1$, we
 deduce that for any $\gamma >0$, there are $\beta>0$ and a subset $D$ of $\Lambda_1$ with meas$D <\pi/6$, such that
    $$f\left(s_1-x,x, \beta\right) \geq \gamma,\ \ \ {\rm for}\ x\in D^c.
    $$
Hence, for $N_1 >2(\beta+d_1)$, we infer from the above inequality and (h4) that
 \begin{eqnarray}
    \int_{\Lambda_1 \cap D^c} f\left(s_1-x,x, \frac{N_1}{2}-d_1 -1\right) \dif x \geq \gamma\; {\rm meas}(\Lambda_1 \cap D^c) \geq \frac{1}{6}\gamma\pi.
 \end{eqnarray}

  On the other hand, by (h2), (h4) and $N_1 >2d_1 +2$, we have
  \begin{eqnarray}
    \int_{\Lambda_1 \cap D} f\left(s_1-x,x, \frac{N_1}{2}-d_1 -1\right) \dif x \geq 0.
 \end{eqnarray}
 Then, with the aid of (5.13), (5.16) and (5.17), we show that
 \begin{eqnarray*}
   \frac{1}{6}\gamma\pi \leq \int_{\Lambda_1} f\left(s_1-x,x, \frac{N_1}{2}-d_1 -1\right) \dif x \leq d,
  \end{eqnarray*}
  which is impossible when we fix $\gamma >6d/\pi$.\par
   Therefore, we conclude that $y$ is essentially bounded. With a similar argument, we have $z\in L^\infty$. The proof of Theorem 1.4 is thereby completed.\hfill $\Box$

\section{Proof of Theorem 1.5}
\setcounter{equation}{0}
 \label{}\ \ \ \  This section is devoted to the proof of the continuity of the solutions for (1.1)$_{a,b,c}$.\par\medskip

 Let $(u, v)\in L^\infty(\Omega)\times L^\infty(\Omega)$ be a solution of (1.1)$_{a,b,c}$ provided by Theorem 1.4. Splitting it into  $(u,v)=(u_1 +y, v_1 +z)$, with $u_1$, $v_1\in H_{b}^+ \oplus H_{b}^- \equiv (\ker L)^\bot$, and $y$, $z\in H^0\equiv \ker L$.
 \subsection{Continuity of the regular terms}
 \label{}\ \ \ \   It follows from (h1) and the fact of $u, v\in L^\infty$ that $bu+\varepsilon v+f(t,x,u) \in L^\infty$ and $bv+\varepsilon
 u+g(t,x,v) \in L^\infty$. Hence, noting that $(u,v)=(u_1 +y,\; v_1 +z)$ solves the system of (1.1)$_{a,b,c}$, we infer from (5.2) that $u_1, v_1 \in C^{0,1}(\Omega)$, which means
  \begin{eqnarray}
   |u_1(t,x)-u_1(\tau,\zeta)| + |v_1(t,x)-v_1(\tau,\zeta)|\leq C\left(|t-\tau|+|x-\zeta|\right),
  \end{eqnarray}
 for all $(t,x)$, $(\tau,\zeta)\in \Omega$. Consequently, we achieve that $u_1(t,x)$ and $v_1(t,x)$ are continuous on $\Omega$.
\subsection{Continuity of the null terms}
 \label{}\ \ \ \     We turn to prove that $y$, $z\in C(\Omega)$. Regarding that $y,\ z\in \ker L$ can be represented in the form of
 $y= p(t+x)- p(t-x)$ and $z= q(t+x)- q(t-x)$, where $p,\ q\in L^1 ([0, 2\pi])$ are $2\pi$-periodic and satisfy $\int_0^{2\pi} p(\tau) \dif \tau =\int_0^{2\pi} q(\tau) \dif \tau =0$, then it suffices to show that $p$ and $q$ are continuous, that means
   $$p(t+h)-p(t)\rightarrow 0,\ \ q(t+h)-q(t)\rightarrow 0,\ \ {\rm as}\ h\rightarrow 0,\ {\rm for\ a.e.}\ t \in [0.2\pi].
   $$

   To reach our goal, we denote by $\hat {p}_{h}(t)=p(t +h)-p(t)$, $\hat {q}_{h}(t)=q(t +h)-q(t)$, for fixed $|h|<1/4$, and let
     $$M_1 = \|\hat {p}_{h}\|_{L^\infty (\Omega)},\ \ M_2 = \|\hat {q}_{h}\|_{L^\infty (\Omega)}.
     $$
Then $M_1\leq 2N_1<\infty$ and $M_2\leq 2N_2<\infty$, where $N_1=\|p\|_{L^\infty (\Omega)}$, and $N_2=\|q\|_{L^\infty (\Omega)}$. Moreover, without loss of generality, we may assume that there exist  $s_1$ and $s_2$, such that
     \begin{eqnarray}
       \hat {p}_{h}(s_1)> M_1(1-|h|),\ \ \hat {q}_{h}(s_2)> M_2(1-|h|).
     \end{eqnarray}
      We intend to prove $M_1\rightarrow0$, $M_2\rightarrow0$, as $|h|\rightarrow 0$.\par
      For simplicity, we denote by $\hat{p}(t)=\hat{p}_{h}(t)$ and $\hat{q}(t)=\hat{q}_{h}(t)$.
       The following notations and estimates are in order. Define
\begin{eqnarray*}
     \phi(\tau)=\min_{\substack{(t,x)\in \Omega\\|s|\leq 2N_1}} \big[f(t,x,u_1(t,x)+s+\tau)-f(t,x,u_1(t,x)+s)\big],
   \end{eqnarray*}
  \begin{eqnarray*}
     \psi(\tau)=\min_{\substack{(t,x)\in \Omega\\|s|\leq 2N_2}} \big[g(t,x,v_1(t,x)+s+\tau)-g(t,x,v_1(t,x)+s)\big],
   \end{eqnarray*}
We shall prove that
   \begin{eqnarray}
    \int_0^\pi \big[\phi(\hat{p}(s_1)-\hat{p}(s_1-2x))+\psi(\hat{q}(s_2)-\hat{q}(s_2-2x))\big]\dif x\leq C|h|,
  \end{eqnarray}
and
\begin{eqnarray}
    \int_{\Sigma_1} \phi(\hat{p}(s_1)-\hat{p}(s_1-2x))\dif x +  \int_{\Sigma_2}\psi(\hat{q}(s_2)-\hat{q}(s_2-2x))\dif x \leq C|h|,
  \end{eqnarray}
where $C>0$ is a number independent of $h$, and
 $$\Sigma_1 = \left\{x\in [0,\pi]:\hat{p}(s_1)-\hat{p}(s_1 -2x)\geq \frac{M_1}{2}\right\},\
  \Sigma_2 = \left\{x\in [0,\pi]:\hat{q}(s_2)-\hat{q}(s_2 -2x)\geq \frac{M_2}{2}\right\}.
 $$
 \subsubsection{Proof of (6.3)}
 \label{}\ \ \ \ We carry out the argument by three steps.\par
 {\textbf{S{\scriptsize{TEP}} 1:} We establish the following lemma to study the behavior of the integral
 \begin{eqnarray*}
 J &=& \int_0^\pi \big[ b\hat{p}(s_1)-b\hat{p}(s_1-2x))+ \varepsilon \hat{q}(s_1)- \varepsilon \hat{q}(s_1-2x))- \tilde{f}(s_1 -x, x)\\
 &&\qquad +\ f\big(s_1-x,x,u_1(s_1-x,x)+p(s_1+h)-p(s_1+h-2x)\big) \big]\dif x.
\end{eqnarray*}
 \begin{lem}
     Assume that $(u(t,x),v(t,x))$ is a solution of {\rm (1.1)}$_{a,b,c}$ obtained in {\rm Theorem 1.1}, then there exists a number $C>0$ independent of $h$, such that
      \begin{eqnarray}
    J \leq  C|h|-\varepsilon \pi \hat{q}(s_1).
  \end{eqnarray}
 \end{lem}
 \begin{proof}
 We denote by
 \begin{eqnarray*}
      J_1&=& \int_0^\pi \big[bu_1(s_1-x,x)+bp(s_1+h)-bp(s_1+h-2x)+\varepsilon v_1(s_1-x,x)+\varepsilon q(s_1+h)\\
      &&\  -\ \varepsilon q(s_1+h-2x)+f\left(s_1-x,x,u_1(s_1-x,x)+p(s_1+h)-p(s_1+h-2x)\right)\big]\dif x,
    \end{eqnarray*}
and
 \begin{eqnarray*}
      J_2&=& \int_0^\pi \big[bu_1(s_1+x,x)+bp(s_1+h+2x)-bp(s_1+h)+\varepsilon v_1(s_1+x,x)+\varepsilon q(s_1+h+2x)\\
      &&\  -\ \varepsilon q(s_1+h)+f\left(s_1+x,x,u_1(s_1+x,x)+p(s_1+h+2x)-p(s_1+h)\right)\big]\dif x.
    \end{eqnarray*}

Since $(u,v)$ is a solution of $Lu + h_1(t,x)=0$, and recalling the notation of the function $h_1(t,x)$ occurs in (5.8), then we note that
    \begin{eqnarray*}
     \int_0^\pi h_1(s_1+h-x,x)\dif x &=& \int_0^\pi h_1(s_1+h+x,x)\dif x\ \ {\rm and}\\
     \int_0^\pi h_1(s_1-x,x)\dif x &=& \int_0^\pi h_1(s_1+x,x)\dif x,
    \end{eqnarray*}
by taking into account (5.4) with $t=s_1+h$ and $t=s_1$ respectively. Then we have
   \begin{eqnarray}
     J &=&\left( J_1 - \int_0^\pi h_1(s_1+h-x,x)\dif x \right) + \left( \int_0^\pi h_1(s_1+h+x,x)\dif x - J_2 \right)\nonumber\\
      &&\ +\ \left( J_2 - \int_0^\pi h_1(s_1+x,x)\dif x \right)\\[0.5em]
      &:=& \tilde{J_1} + \tilde{J_2}+ \tilde{J_3}.\nonumber
   \end{eqnarray}

 \textit{Behavior of $\tilde{J_1}$ and} $\tilde{J_2}$:\par
  We begin with the study of the asymptotic behavior of the first term in the right hand side of (6.6). It is plain to check that
  \begin{eqnarray*}
     \tilde{J_1}&=& \int_0^\pi \big[bu_1(s_1-x,x)-bu_1(s_1+h-x,x)+\varepsilon v_1(s_1-x,x)-\varepsilon v_1(s_1+h-x,x)\\
      &&\  -\ f\left(s_1+h-x,x,u_1(s_1+h-x,x)+p(s_1+h)-p(s_1+h-2x)\right)\\
      &&\  +\ f\left(s_1-x,x,u_1(s_1-x,x)+p(s_1+h)-p(s_1+h-2x)\right)\big]\dif x.
    \end{eqnarray*}

   By virtue of (6.1) and $f\in C^1(\Omega\times {\mathbb R})$, we get
  \begin{eqnarray}
  |\tilde{J_1}|\leq \int_0^\pi \big[C|h|+C|u_1(s_1-x,x)-u_1(s_1+h-x,x)|\big]\dif x\leq C|h|,
   \end{eqnarray}
where $C>0$ is independent of $h$. Moreover, a similar calculation shows that
\begin{eqnarray}
      |\tilde{J_2}|&\leq& \int_0^\pi \Huge|bu_1(s_1+h+x,x)-bu_1(s_1+x,x)+\varepsilon v_1(s_1+h+x,x)-\varepsilon v_1(s_1+x,x)\nonumber\\
      &&\  +\ f\left(s_1+h+x,x,u_1(s_1+h+x,x)+p(s_1+h+2x)-p(s_1+h)\right)\nonumber\\
      &&\  -\ f\left(s_1+x,x,u_1(s_1+x,x)+p(s_1+h+2x)-p(s_1+h)\right)\Huge|\dif x\nonumber\\
      &\leq& C|h|.
    \end{eqnarray}

 \textit{Behavior of} $\tilde{J_3}$:\par
  Using the notations of $\hat {p}(t)=p(t +h)-p(t)$ and $\hat {q}(t)=q(t +h)-q(t)$, we can represent
  \begin{eqnarray*}
   \tilde{J_3} &=& \int_0^\pi \big[b\hat{p}(s_1+2x)-b\hat{p}(s_1)+\varepsilon \hat{q}(s_1+2x)-\varepsilon \hat{q}(s_1)\\
      &&\  +\ f\left(s_1+x,x,u_1(s_1+x,x)+p(s_1+h+2x)-p(s_1+h)\right)\\
      &&\  -\ f\left(s_1+x,x,u_1(s_1+x,x)+p(s_1+2x)-p(s_1)\right)\big]\dif x.
  \end{eqnarray*}

   By (6.2) and the definition of $M_1$, we obtain
   $$\hat{p}(s_1+2x) \leq M_1 < \hat{p}(s_1)+M_1|h|,
   $$
 which indicates that
   \begin{eqnarray}
     p(s_1+h+2x)-p(s_1+h)<p(s_1+2x)-p(s_1)+M_1|h|.
   \end{eqnarray}

   In addition, from the facts of $q$ is $2\pi$-periodic and $\int_0^{2\pi} q(\tau) \dif \tau =0$, we deduce that
  $$\int_0^{\pi} \hat{q}(s_1+2x) \dif x = \int_0^{\pi}\big[ q(s_1+h+2x)-q(s_1+2x)\big] \dif x =0.
  $$
Thus, the above estimates give
\begin{eqnarray}
  \int_0^\pi \big[b\hat{p}(s_1+2x)-b\hat{p}(s_1)+\varepsilon \hat{q}(s_1+2x)-\varepsilon \hat{q}(s_1)\big]\dif x \leq
  \pi bM_1|h|-\varepsilon \pi \hat{q}(s_1).
\end{eqnarray}
  Moreover, we infer from (6.9), (h4) and  $f\in C^1(\Omega\times {\mathbb R})$ that
  \begin{eqnarray}
   &&\int_0^\pi \big[\ f\left(s_1+x,x,u_1(s_1+x,x)+p(s_1+h+2x)-p(s_1+h)\right) \nonumber\\
      &&\ \ \ \ \ -\ f\left(s_1+x,x,u_1(s_1+x,x)+p(s_1+2x)-p(s_1)\right)\big]\dif x \nonumber\\
   &\leq& CM_1 \pi |h|.
  \end{eqnarray}
Therefore, making use of (6.10) and (6.11), we have
   \begin{eqnarray}
  \tilde{J_3} \leq  \pi bM_1|h|-\varepsilon \pi \hat{q}(s_1)+CM_1\pi |h|.
\end{eqnarray}

   Now, inserting the estimates (6.7), (6.8) and (6.12) into (6.6), we obtain (6.5).
\end{proof}
     \par\medskip
 {\textbf{S{\scriptsize{TEP}} 2:}} We require the next lemma concerning with the properties of the integrals
  \begin{eqnarray*}
    R_1&=&\int_0^\pi \big[f\left(s_1-x,x,u_1(s_1-x,x)+p(s_1+h)-p(s_1+h-2x)\right)\\
        &&\ \ -\ f\left(s_1-x,x,u_1(s_1-x,x)+p(s_1)-p(s_1-2x)\right)\big]\dif x,
  \end{eqnarray*}
  \begin{eqnarray*}
    R_2&=&\int_0^\pi \big[g\left(s_2-x,x,v_1(s_2-x,x)+q(s_2+h)-q(s_2+h-2x)\right)\\
        &&\ \ -\ g\left(s_2-x,x,v_1(s_2-x,x)+q(s_2)-q(s_2-2x)\right)\big]\dif x.
  \end{eqnarray*}
  \begin{lem}
  Under the assumptions of {\rm Theorem 1.5}, we have $R_1+R_2\leq C|h|$ for $\varepsilon$ is sufficiently small, where $C>0$ is independent of $h$.
  \end{lem}
\begin{proof}
  Noting that $\int_0^{\pi} \hat{p}(s_1-2x) \dif x = \int_0^{\pi} \hat{q}(s_1-2x) \dif x =0$, then by a direct computation we observe that
  \begin{eqnarray*}
  R_1 &=& J - \int_0^\pi \big[b\hat{p}(s_1) -b\hat{p}(s_1-2x) +\varepsilon \hat{q}(s_1) -\varepsilon \hat{q}(s_1-2x)\big]\dif x\\
     &=& J - b\pi \hat{p}(s_1)- \varepsilon \pi \hat{q}(s_1).
  \end{eqnarray*}
 Hence, with the help of Lemma 6.1 we have
  \begin{eqnarray}
    R_1\leq -b\pi \hat{p}(s_1) - 2\varepsilon \pi \hat{q}(s_1) + C|h|.
  \end{eqnarray}
  Similar to the derivation of the inequalities (6.5) and (6.13), we are able to get
  \begin{eqnarray}
    R_2\leq -b\pi \hat{q}(s_2) - 2\varepsilon \pi \hat{p}(s_2) + C|h|.
  \end{eqnarray}

  Selecting $|\varepsilon| <b/2$, we deduce from (6.2) and the definitions of $M_1$, $M_2$ that
  $$b \hat{p}(s_1)+2\varepsilon \hat{p}(s_2)\geq bM_1(1-|h|)-2|\varepsilon| M_1 >-bM_1|h|,
  $$
  $$b \hat{q}(s_2)+2\varepsilon \hat{q}(s_1)\geq bM_2(1-|h|)-2|\varepsilon| M_2 >-bM_2|h|.
  $$

  By virtue of the above two inequalities, then adding up (6.13) and (6.14) will yield that $R_1+R_2\leq C|h|$ for $|\varepsilon|<b/2$. That is what we desire.
  \end{proof}

 {\textbf{S{\scriptsize{TEP}} 3:  Now, we are in a position to prove (6.3).}}\par
  Under the assumptions of Theorem 1.5, we know $\phi(\tau)$, $\psi(\tau)$ are strictly increasing in $\tau$ and
$\phi(0)=0$, $\psi(0)=0$. Moreover, it follows from $f$, $g\in C^1$ that $\phi(\tau)$, $\psi(\tau)$ are Lipschitz continuous on any bounded intervals. \par
  By the definitions of $R_1$, $R_2$, $\phi$ and $\psi$, we conclude that
  \begin{eqnarray*}
   \phi(\hat{p}(s_1)-\hat{p}(s_1-2x)) &\leq& f(s_1-x, x, u_1(s_1-x, x)+p(s_1+h)-p(s_1-2x+h))\\
            &&-\ f(s_1-x, x, u_1(s_1-x, x)+p(s_1)-p(s_1-2x)).
  \end{eqnarray*}
Then integrating the above in $x$ leads to
  $$\int_0^\pi \phi(\hat{p}(s_1)-\hat{p}(s_1-2x))\dif x\leq R_1.
  $$
And a similar computation shows that $\int_0^\pi \psi(\hat{q}(s_2)-\hat{q}(s_2-2x))\dif x\leq R_2$.\par

   Therefore, we derive (6.3) from Lemma 6.2.\hfill $\Box$
    \subsubsection{Proof of (6.4)}

 \label{}\ \ \ \ We denote by  $\Sigma_1^c$, $\Sigma_2^c$ the complement spaces of $\Sigma_1$, $\Sigma_2$ respectively.\par
  For any $x\in [0,\pi]$, it follows from (6.2) and the choice of $M_1$ that
  $$\hat{p}(s_1)-\hat{p}(s_1 -2x)\geq M_1(1-|h|)-M_1= -M_1|h|.
  $$
Combining the above inequality with the facts of $\phi(0)=0$, $\phi$ is strictly increasing and Lipschitz continuous, we conclude that
  \begin{eqnarray*}
    \phi(\hat{p}(s_1)-\hat{p}(s_1-2x)) \geq \phi(-M_1|h|)-\phi(0)\geq -C|h|,\ \ {\rm for}\ x\in \Sigma_1^c.
  \end{eqnarray*}
Integrating the above in $x$ on $\Sigma_1^c$, we have
  \begin{eqnarray}
    -\int_{\Sigma_1^c} \phi(\hat{p}(s_1)-\hat{p}(s_1-2x))\dif x \leq C|h|.
  \end{eqnarray}

  On the other hand, concerning with the behavior of $\psi(\hat{q}(s_2)-\hat{q}(s_2-2x))$ restricted on $\Sigma_2^c$, we also get
  \begin{eqnarray}
    -\int_{\Sigma_2^c} \psi(\hat{q}(s_2)-\hat{q}(s_2-2x))\dif x \leq C|h|.
  \end{eqnarray}

  Hence, the inequalities of (6.3), (6.15) and (6.16) ensure that
  \begin{eqnarray*}
    &&\int_{\Sigma_1} \phi(\hat{p}(s_1)-\hat{p}(s_1-2x))\dif x +  \int_{\Sigma_2}\psi(\hat{q}(s_2)-\hat{q}(s_2-2x))\dif x\nonumber\\
    &=& \left(\int_0^\pi - \int_{\Sigma_1^c}\right)\phi(\hat{p}(s_1)-\hat{p}(s_1-2x))\dif x
     + \left(\int_0^\pi -\int_{\Sigma_2^c}\right)\psi(\hat{q}(s_2)-\hat{q}(s_2-2x))\dif x\nonumber\\[0.3em]
    &\leq& C|h|,
  \end{eqnarray*}
where $C>0$ is independent of $h$. Thus, we arrive at (6.4).\hfill $\Box$
\subsubsection{Complete the proof of Theorem 1.5}
 \label{}\ \ \ \ \  Finally, we examine the continuity of $p$ and $q$ by controlling the values of $\phi(M_1/2)$ and $\psi(M_2/2)$.\par

  Recalling the definitions of $\Sigma_1$, $\Sigma_2$ and the facts that $\phi(\tau)$, $\psi(\tau)$ are strictly increasing in
$\tau$, we deduce that $\phi(M_1/2)\leq \phi(\hat{p}(s_1)-\hat{p}(s_1-2x))$ for $x\in \Sigma_1$, and $\psi(M_2/2)\leq \phi(\hat{q}(s_2)-\hat{q}(s_2-2x))$ for $x\in \Sigma_2$. Thus, we infer from (6.4) that
\begin{eqnarray}
  \int_{\Sigma_1} \phi(\frac{M_1}{2})\dif x +  \int_{\Sigma_2}\psi(\frac{M_2}{2})\dif x \leq C|h|.
\end{eqnarray}

  We next claim that meas$\Sigma_1>0$ and meas$\Sigma_2>0$.\par
  Indeed, we follow from $\int_{0}^{2\pi}\hat{p}(\tau) \dif \tau=0$ that
   \begin{eqnarray}
     0=\int_{0}^{\pi}\hat{p}(s_1-2x) \dif x=\int_{\Sigma_1}\hat{p}(s_1-2x) \dif x + \int_{\Sigma_1^c}\hat{p}(s_1-2x) \dif x.
   \end{eqnarray}
 On the other hand, according to the definitions of $M_1$, $s_1$ and $\Sigma_1^c$, we deduce
  \begin{eqnarray}
    \hat{p}(s_1-2x) \geq \left\{\begin{array}{l}
      -M_1,\ \ {\rm for}\ x\in \Sigma_1,
    \\[1.2ex]
     \hat{p}(s_1)-\frac{M_1}{2} > \frac{M_1}{2}-M_1|h|,\ \ {\rm for}\ x\in \Sigma_1^c.
 \end{array} \right.
  \end{eqnarray}
Then, comparing with (6.18) and (6.19), we arrive at
   \begin{eqnarray*}
     0\geq -M_1 {\rm meas}\Sigma_1 +(\frac{1}{2}-|h|)M_1(\pi-{\rm meas}\Sigma_1)
     = -(\frac{3}{2}-|h|)M_1{\rm meas}\Sigma_1 + (\frac{1}{2}-|h|)M_1\pi,
   \end{eqnarray*}
 which implies that
    \begin{eqnarray*}
      {\rm meas}\Sigma_1>\frac{(1-2|h|)\pi}{3-2|h|}>\frac{\pi}{5}, \ \ {\rm as}\ |h|<\frac{1}{4}.
    \end{eqnarray*}

    A similar argument allows us to obtain ${\rm meas}\Sigma_2>\pi/5$, for $|h|<1/4$.\par
     Therefore, we make use of (6.17) and the preceding estimates for meas$\Sigma_1$, meas$\Sigma_2$ to get
    $$\phi(\frac{M_1}{2})+\psi(\frac{M_2}{2}) \leq C|h|.
    $$

  Furthermore, noting that $\phi(M_1/2)>0$, $\psi(M_2/2)>0$ and $\phi(\tau)$, $\psi(\tau)$ are strictly increasing in
$\tau$, we conclude that $M_1\rightarrow 0$ and $M_2\rightarrow 0$.
  Hence, it follows that $y= p(t+x)- p(t-x)\in C(\Omega)$ and $z= q(t+x)- q(t-x)\in C(\Omega)$. We finish the proof of Theorem 1.5.\hfill $\Box$


\end{document}